\documentclass[final,1p,times,authoryear]{elsarticle}

\usepackage{amsmath}
\usepackage{amssymb}
\usepackage{enumerate}
\usepackage{mathrsfs}
\usepackage{color}
\usepackage{todonotes}
\usepackage{tikz}
\usepackage{pgfplots}
\pgfplotsset{width=8cm,height=4cm,compat=1.9}
\usepackage{algorithmic,algorithm}
\usepackage{multirow}
\usepackage{graphicx}
\usepackage{subfigure}
\usepackage{caption}
\definecolor{darkblue}{rgb}{0.0, 0.0, 0.55}
\definecolor{bordeaux}{rgb}{0.34, 0.01, 0.1}
\usepackage[pagebackref,colorlinks,linkcolor=bordeaux,citecolor=darkblue,urlcolor=black,hypertexnames=true]{hyperref}

\newtheorem{theorem}{Theorem}
\newtheorem{lemma}[theorem]{Lemma}

\newtheorem{conjecture}[theorem]{Conjecture}

\newtheorem{definition}[theorem]{Definition}

\newtheorem{example}[theorem]{Example}
\newtheorem{remark}[theorem]{Remark}

\newenvironment{proof}{\noindent{\em Proof:}}{$\Box$~\\}
\def\Z{{\mathbb{Z}}}
\def\Q{{\mathbb{Q}}}
\def\R{{\mathbb{R}}}

\def\N{{\mathbb{N}}}

\def\x{{\mathbf{x}}}

\def\a{{\boldsymbol{\alpha}}}

\def\b{{\boldsymbol{\beta}}}
\def\gg{{\boldsymbol{\gamma}}}
\def\ba{{\boldsymbol{a}}}
\def\bv{{\boldsymbol{v}}}
\def\bu{{\boldsymbol{u}}}
\def\bw{{\boldsymbol{w}}}
\def\A{{\mathscr{A}}}
\def\B{{\mathscr{B}}}
\def\CC{{\mathscr{C}}}
\def\TT{{\mathscr{T}}}
\def\supp{\hbox{\rm{supp}}}

\def\SONC{\hbox{\rm{SONC}}}

\def\int{\hbox{\rm{int}}}
\def\New{\hbox{\rm{New}}}
\def\Conv{\hbox{\rm{conv}}}

\DeclareMathOperator{\bigo}{\mathcal{O}}
\DeclareMathOperator{\bigotilde}{\widetilde{\mathcal{O}}}

\newcommand{\bm}[1]{\mbox{\boldmath{$#1$}}}
\newcommand{\revision}[1]{{{\color{black}#1}}}

\begin{document}
\begin{frontmatter}
\title{SONC Optimization and Exact Nonnegativity Certificates via Second-Order Cone Programming}

\author{Victor Magron}
\address{CNRS LAAS, 7 avenue du colonel Roche 31400 Toulouse; Institute of Mathematics, 31400 Toulouse France.}
\ead{vmagron@laas.fr}
\ead[url]{https://homepages.laas.fr/vmagron}

\author{Jie Wang}
\address{CNRS LAAS, 7 avenue du colonel Roche 31400 Toulouse France}
\ead{jwang@laas.fr}
\ead[url]{https://wangjie212.github.io/jiewang}

\begin{abstract}
The second-order cone (SOC) is a class of simple convex cones and optimizing over them can be done more efficiently than with semidefinite programming. 
It is interesting both in theory and in practice to investigate which convex cones admit a representation using SOCs, given that they have a strong expressive ability. 
In this paper, we prove constructively that the cone of sums of nonnegative circuits (SONC) admits a SOC representation. 
Based on this, we give a new algorithm for unconstrained polynomial optimization via SOC programming.
We also provide a hybrid numeric-symbolic scheme which combines the numerical procedure with a rounding-projection algorithm to obtain exact nonnegativity certificates.
Numerical experiments demonstrate the efficiency of our algorithm for polynomials with fairly large degree and number of variables.
\end{abstract}

\if{

}\fi

\end{frontmatter}

\section{Introduction}
A {\em circuit polynomial} is of the form
$\sum_{\a\in\TT}c_{\a}\x^{\a}-d\x^{\b}\in\R[\x]=\R[x_1,\ldots,x_n],$
where $c_{\a}>0$ for all $\a\in\TT$, $\TT\subseteq(2\N)^n$ comprises the vertices of a simplex and $\b$ lies in the relative interior of this simplex.
The set of {\em sums of nonnegative circuit polynomials (SONC)} was introduced by \cite{iw} as a new certificate of nonnegativity for sparse polynomials, which is independent of the well-known set of sums of squares (SOS). 
Another recently introduced alternative certificates by \cite{ca} are sums of arithmetic-geometric-exponentials (SAGE), which can be obtained via relative entropy programming.
The connection between SONC and SAGE polynomials has been recently studied in \cite{mu,wang,ka19}. 
It happens that SONC polynomials and SAGE polynomials are actually equivalent, as proved by \cite{wang} and \cite{mu}, and that both have a cancellation-free representation in terms of generators; see prior research by \cite{wang,mu}.
 
One of the significant differences between SONC and SOS is that SONC decompositions preserve sparsity of polynomials while SOS decompositions of sparse polynomials are not necessarily sparse; see \cite{wang}. 
The set of SONC polynomials with a given support forms a convex cone, called a {\em SONC cone}. Optimization problems over SONC cones can be formulated as geometric programs or more generally relative entropy programs. 
We refer the interested reader to \cite{lw} for the unconstrained case and to \cite{diw} for the constrained case. 
Numerical experiments for unconstrained POPs (polynomial optimization problems) in \cite{se} have demonstrated the advantage of the SONC-based methods compared to the SOS-based methods, especially in the high-degree but fairly sparse case.
Another recent framework by \cite{dressler2020global} relies on the dual SONC cone to compute lower bounds of  polynomials by means of linear programming instead of geometric programming. 
As emphasized in the numerical comparison from   \cite{dressler2020global}, there are no general guarantees that the bounds obtained with this framework are always more or less accurate than the approach based on geometric programming from \cite{se}, and the same holds for performance.

In the SOS case, there have been several attempts to exploit sparsity occurring in POPs. 
The sparse variant developed by \cite{Waki06SparseSOS} of the moment-SOS hierarchy exploits the correlative sparsity pattern among the input variables to reduce the support of the resulting SOS decompositions.
Such sparse representation results have been successfully applied in many fields, such as optimal power flow by \cite{Josz16}, 
roundoff error bounds by \cite{ma17} and recently extended to the noncommutative case by \cite{ncsparse}.
Another way to exploit sparsity is to consider patterns based on terms, rather than variables, yielding an alternative sparse variant of Lasserre's hierarchy. 
The resulting \emph{term-sparsity} SOS (TSSOS) hierarchy has been developed for polynomial optimization by \cite{tssos,tssos2,wang2020sparsejsr}, combined with correlative sparsity in \cite{cstssos}, and extended to the noncommutative case in \cite{nctssos}.

One of the similar features shared by SOS/SONC-based frameworks is their intrinsic connections with conic programming: SOS decompositions are computed via semidefinite programming and SONC decompositions via geometric programming. 
In both cases, the resulting optimization problems are solved with interior-point algorithms, thus output approximate nonnegativity certificates. 
However, one can still obtain an exact certificate from such output via hybrid numerical-symbolic algorithms when the input polynomial lies in the interior of the SOS/SONC cone.
One way is to rely on rounding-projection algorithms adapted to the SOS cone by \cite{pe} and the SONC cone by \cite{ma19}, or alternatively on perturbation-compensation schemes as in \cite{univsos,multivsos18}. 
The latter algorithms are implemented within the {\tt RealCertify} library by \cite{RealCertify}.

In this paper, we study the second-order cone representation of SONC cones. An $n$-dimensional {\em (rotated) second-order cone (SOC)} is defined as
\begin{equation*}
\mathbf{K}:=\{\ba=(a_i)_{i=1}^n\in\R^{n}\mid2a_1a_2\ge\sum_{i=3}^{n}a_i^2,a_1\ge0,a_2\ge0\}.
\end{equation*}
The SOC is well-studied and has mature solvers. Optimizing via second-order cone programming (SOCP) can be handled more efficiently than with semidefinite programming; see the work by \cite{ah} and \cite{al}. On the other hand, despite the simplicity of SOCs, they have a strong ability to express other convex cones. 
Many such examples can be found in Section 3.3 from \cite{ben2001lectures}. 
Therefore, it is interesting in theory and also important from the view of applications to investigate which convex cones can be expressed by SOCs. 

Given sets of lattice points $\A\subseteq(2\N)^n$, $\B_1\subseteq\Conv(\A)\cap(2\N)^n$ and $\B_2\subseteq\Conv(\A)\cap(\N^n\backslash(2\N)^n)$ ($\Conv(\A)$ is the convex hull of $\A$) with $\A\cap\B_1=\varnothing$, let $\SONC_{\A,\B_1,\B_2}$ be the SONC cone supported on $\A,\B_1,\B_2$ (see Definition \ref{def-cone}).  
The first main result of this paper is the following theorem.
\begin{theorem}\label{mthm}
For $\A\subseteq(2\N)^n$, $\B_1\subseteq\Conv(\A)\cap(2\N)^n$ and $\B_2\subseteq\Conv(\A)\cap(\N^n\backslash(2\N)^n)$ with $\A\cap\B_1=\varnothing$, the convex cone $\SONC_{\A,\B_1,\B_2}$ admits a SOC representation.
\end{theorem}

The fact that SONC cones admit a SOC characterization was firstly proven by \cite{abe} (see Theorem 17 for more details). However, Averkov's result is more theoretical. Even though Averkov's proof theoretically allows one to construct a SOC representation for a SONC cone, the construction is complicated and was not explicitly given in Averkov's paper. Our proof of Theorem \ref{mthm}, which involves writing a SONC polynomial as a sum of binomial squares with rational exponents (Theorem \ref{sec3-thm5}), is totally different from Averkov's and leads to a more concise (hence more efficient) SOC representation for SONC cones. This representation also enables us to propose a practical algorithm, based on SOCP, providing SONC decompositions for a certain class of nonnegative polynomials, which in turn yields lower bounds for unconstrained POPs. We emphasize that the number of SOCs used in the algorithm is {\em linear} in the number of terms (Lemma \ref{bound-triple}).
We test the algorithm on various randomly generated polynomials up to a fairly large size, involving  $n \sim 40$ variables and of degree $d \sim 60$. The numerical results demonstrate the efficiency of our algorithm.

The rest of this paper is organized as follows. In Section \ref{sec:prelim}, we list some preliminaries on SONC polynomials. 
In Section \ref{sec:sobs}, we reveal a key connection between SONC polynomials and sums of binomial squares by introducing the notion of $\TT$-rational mediated sets. 
By virtue of this connection, we obtain second-order cone representations for SONC cones in Section \ref{sec:socrepr}. 
In Section \ref{sec:optim}, we provide a numerical procedure 
for unconstrained polynomial optimization via SOCP. 
A hybrid numerical-symbolic certification algorithm, based on combining this numerical procedure with a symbolic rounding-projection scheme is given and analyzed in Section \ref{sec:exact}.
We evaluate the performance of the resulting certification algorithm in Section \ref{sec:benchs}.

This paper is the follow-up of our previous contribution~\citep{soncsocp}, published in the proceedings of ISSAC'20. 
The main theoretical and practical novelties are threefold: (1) we design a hybrid numeric-symbolic algorithm to compute exact nonnegativity certificates, (2) provide explicit bounds for its arithmetic complexity analysis and (3) compare its efficiency with our numerical optimization procedure.
Another more minor novelty is that we provide detailed statements for the algorithms to build rational mediated sets and simplex covers, respectively in Section \ref{sec:ratmediatedsets} and Section \ref{sec:optim}.
%

\section{Preliminaries}
\label{sec:prelim}
Let $\R[\x]=\R[x_1,\ldots,x_n]$ be the ring of real $n$-variate polynomials, and let $\R_+$ be the set of positive real numbers. For a finite set $\A\subseteq\N^n$, we denote by $\Conv(\A)$ the convex hull of $\A$. Given a finite set $\A\subseteq\N^n$, we consider polynomials $f\in\R[\x]$ supported on $\A\subseteq\N^n$, i.e., $f$ is of the form $f(\x)=\sum_{\a\in \A}f_{\a}\x^{\a}$ with $f_{\a}\in\R, \x^{\a}=x_1^{\alpha_1}\cdots x_n^{\alpha_n}$. The support of $f$ is $\supp(f):=\{\a\in \A\mid f_{\a}\ne0\}$ and the Newton polytope of $f$ is defined as $\New(f):=\Conv(\supp(f))$. For a polytope $P$, we use $V(P)$ to denote the vertex set of $P$ and use $P^{\circ}$ to denote the interior of $P$. For a set $A$, we use $\#A$ to denote the cardinality of $A$.
A polynomial $f\in\R[\x]$ which is nonnegative over $\R^n$ is called a {\em nonnegative polynomial}, or a {\em positive semi-definite (PSD) polynomial}.
The following definition of circuit polynomials was proposed by Iliman and De Wolff in \cite{iw}.
\begin{definition}
A polynomial $f\in\R[\x]$ is called a {\em circuit polynomial} if it is of the form
$f(\x)=\sum_{\a\in\TT}c_{\a}\x^{\a}-d\x^{\b}$
and satisfies the following conditions: 
\begin{enumerate}[(i)]
    \item $\TT\subseteq(2\N)^n$ comprises the vertices of a simplex;
    \item $c_{\a}>0$ for each $\a\in\TT$;
    \item $\b\in\Conv(\TT)^{\circ}\cap\N^n$.
\end{enumerate}
The support $(\TT, \b)$ satisfying $(i),(iii)$ is called a {\em circuit}.
\end{definition}

If $f=\sum_{\a\in\TT}c_{\a}\x^{\a}-d\x^{\b}$ is a circuit polynomial, then from the definition we can uniquely write
$\b=\sum_{\a\in\TT}\lambda_{\a}\a\textrm{ with } \lambda_{\a}>0 \textrm{ and } \sum_{\a\in\TT}\lambda_{\a}=1.$
We define the corresponding {\em circuit number} as $\Theta_f:=\prod_{\a\in\TT}(c_{\a}/\lambda_{\a})^{\lambda_{\a}}$. 
The nonnegativity of the circuit polynomial $f$ is decided by its circuit number alone, that is, $f$ is nonnegative if and only if either $\b\notin(2\N)^n$ and $|d|\le\Theta_f$, or $\b\in(2\N)^n$ and $d\le\Theta_f$, by Theorem 3.8 from \cite{iw}.
%
To provide a concise narrative, we refer to a nonnegative circuit polynomial by a nonnegative circuit and also view a monomial square as a nonnegative circuit.
An explicit representation of a polynomial being a {\em sum of nonnegative circuits}, or {\em SONC} for short, provides a certificate for its nonnegativity. Such a certificate is called a {\em SONC decomposition}. A polynomial that admits a SONC decomposition is called a {\em SONC polynomial}. For simplicity, we denote the set of SONC polynomials by SONC.

For a polynomial $f=\sum_{\a\in\A}f_{\a}\x^{\a}\in\R[\x]$, let
$\Lambda(f):=\{\a\in\A\mid\a\in(2\N)^n\textrm{ and }f_{\a}>0\}$
and $\Gamma(f):=\supp(f)\backslash\Lambda(f)$. Then we can write $f$ as
$f=\sum_{\a\in\Lambda(f)}c_{\a}\x^{\a}-\sum_{\b\in\Gamma(f)}d_{\b}\x^{\b}.$
For each $\b\in\Gamma(f)$, let
\begin{equation}\label{fb}
\CC(\b):=\{\TT\mid\TT\subseteq\Lambda(f)\textrm{ and }(\TT,\b)\textrm{ is a circuit}\}.
\end{equation}
As a consequence of Theorem 5.5 from \cite{wang}, if $f\in\SONC$ then it has a  decomposition
\begin{equation}\label{sec2-eq1}
f=\sum_{\b\in\Gamma(f)}\sum_{\TT\in\CC(\b)}f_{\TT\b}+\sum_{\a\in \tilde{\A}}c_{\a}\x^{\a},
\end{equation}
where $f_{\TT\b}$ is a nonnegative circuit polynomial supported on $\TT\cup\{\b\}$ and $\tilde{\A}=\{\a\in\Lambda(f)\mid\a\notin\cup_{\b\in\Gamma(f)}\cup_{\TT\in\CC(\b)}\TT\}$.

\section{SONC and sums of binomial squares}
\label{sec:sobs}
In this section, we give a characterization of SONC polynomials in terms of sums of binomial squares (SOBS) with rational exponents.

\subsection{Rational mediated sets}
\label{sec:ratmediatedsets}
A lattice point $\a\in\N^n$ is {\em even} if it is in $(2\N)^n$. For a subset $M\subseteq\N^n$, define
$\overline{A}(M):=\{\frac{1}{2}(\bv+\bw)\mid\bv\ne\bw,\bv,\bw\in M\cap(2\N)^n\}$
as the set of averages of distinct even points in $M$. A subset $\TT\subseteq(2\N)^n$ is called a {\em trellis} if $\TT$ comprises the vertices of a simplex. For a trellis $\TT$, we call $M$ a {\em $\TT$-mediated set} if $\TT\subseteq M\subseteq\overline{A}(M)\cup\TT$;  see the work on mediated sets by \cite{har,pow,re}.
\begin{theorem}\label{sec3-thm1}
Let $f=\sum_{\a\in\TT} c_{\a}\x^{\a}-d\x^{\b}\in\R[\x]$ with $d\ne0$ be a nonnegative circuit. Then $f$ is a sum of binomial squares if and only if there exists a $\TT$-mediated set containing $\b$. Moreover, suppose that $\b$ belongs to a $\TT$-mediated set $M$ and for each $\bu\in M\backslash\TT$, let us write $\bu=\frac{1}{2}(\bv_\bu+\bw_\bu)$ for some $\bv_\bu \ne\bw_\bu \in M\cap(2\N)^n$. 
Then one has the decomposition  $f=\sum_{\bu\in M\backslash\TT}(a_{\bu}\x^{\frac{1}{2}\bv_\bu}-b_{\bu}\x^{\frac{1}{2}\bw_\bu})^2+c\x^{\a}$, with $a_{\bu},b_{\bu}\in\R,c\ge0,\a\in\TT$.
\end{theorem}
\begin{proof}
It follows from Theorem 5.2 in \cite{iw}.
\end{proof}

By Theorem \ref{sec3-thm1}, if we want to represent a nonnegative circuit polynomial as a sum of binomial squares, we need to first decide if there exists a $\TT$-mediated set containing a given lattice point and then to compute one if there exists. However, there are obstacles for each of these two steps: (1) there may not exist such a $\TT$-mediated set containing a given lattice point; (2) even if such a set exists, there is no efficient algorithm to compute it. In order to overcome these two difficulties, we introduce the concept of $\TT$-rational mediated sets as a replacement of $\TT$-mediated sets by admitting rational numbers in coordinates. 

Concretely, for a subset $M\subseteq\Q^n$, let us define
$\widetilde{A}(M):=\{\frac{1}{2}(\bv+\bw)\mid\bv\ne\bw,\bv,\bw\in M\}$
as the set of averages of distinct rational points in $M$. Let us assume that $\TT\subseteq\Q^n$ comprises the vertices of a simplex. We say that $M$ is a {\em $\TT$-rational mediated set} if $\TT\subseteq M\subseteq\widetilde{A}(M)\cup\TT$. We shall see that for a trellis $\TT$ and a lattice point $\b\in\Conv(\TT)^{\circ}$, a $\TT$-rational mediated set containing $\b$ always exists and moreover, there is an effective algorithm to compute it.

First, let us consider the one dimensional case. For a sequence of integer numbers $A=\{s,q_1,\ldots,q_m,p\}$ (arranged from small to large), if every $q_i$ is an average of two distinct numbers in $A$, then we say $A$ is an {\em $(s,p)$-mediated sequence}. Note that the property of $(s,p)$-mediated sequences is preserved under translations, that is, there is a one-to-one correspondence between $(s,p)$-mediated sequences and $(s+r,p+r)$-mediated sequences for any integer number $r$. So it suffices to consider the case of $s=0$.

For a fixed $p$ and an integer $q$ with $0<q<p$, a {\em minimal $(0,p)$-mediated sequence} containing $q$ is a $(0,p)$-mediated sequence containing $q$ with the least number of elements. 
\begin{example}
Consider the set $A=\{0,2,4,5,8,11\}$. One can easily check by hand that $A$ is a minimal $(0,11)$-mediated sequence containing $2,4,5,8$.
\end{example}
Denote the number of elements in a minimal $(0,p)$-mediated sequence containing $q$ by $N(\frac{q}{p})$. One can then easily show that $N(\frac{1}{p})=\left\lceil\log_2(p)\right\rceil+2$ by induction on $p$. We conjecture that this formula holds for general $q$.
\begin{conjecture}\label{con}
If gcd$(p,q) = 1$, then $N(\frac{q}{p})=\left\lceil\log_2(p)\right\rceil+2$.
\end{conjecture}

Generally we do not know how to compute a minimal $(0,p)$-mediated sequence containing a given $q$.
However, we can design a procedure to compute an approximately minimal $(0,p)$-mediated sequence containing a given $q$, based on the following lemma.
\begin{lemma}\label{sec4-lm3}
For integers $p,q$ such that $0<q<p$, there exists a $(0,p)$-mediated sequence containing $q$ with the cardinality less than $\frac{1}{2}(\log_2(p)+\frac{3}{2})^2$.
\end{lemma}
\begin{proof}
We can assume gcd$(p,q) = 1$ (otherwise one can consider $p/$\textrm{gcd}$(p,q),q/$\textrm{gcd}$(p,q)$ instead). Let us do induction on $p$. Assume that for any $p',q'\in\N,0<q'<p'<p$, there exists a $(0,p')$-mediated sequence containing $q'$ with the number of elements less than $\frac{1}{2}(\log_2(p')+\frac{3}{2})^2$.

{\bf Case 1:} Suppose that $p$ is an even number. If $q=\frac{p}{2}$, then by gcd$(p,q) = 1$, we have $q=1$ and $A=\{0,1,2\}$ is a $(0,p)$-mediated sequence containing $q$. Otherwise, we have either $0<q<\frac{p}{2}$ or $\frac{p}{2}<q<p$. For $0<q<\frac{p}{2}$, by the induction hypothesis, there exists a $(0,\frac{p}{2})$-mediated sequence $A'$ containing $q$. For $\frac{p}{2}<q<p$, since the property of mediated sequences is preserved under translations, one can first subtract $\frac{p}{2}$ and obtain a $(0,\frac{p}{2})$-mediated sequence containing $q-\frac{p}{2}$ by the induction hypothesis. Then by adding $\frac{p}{2}$, one obtains a $(\frac{p}{2},p)$-mediated sequence $A'$ containing $q$. It follows that $A=A'\cup\{p\}$ or $A=\{0\}\cup A'$ is a $(0,p)$-mediated sequence containing $q$.
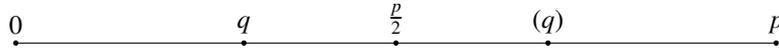
\begin{center}
\begin{tikzpicture}
\draw (0,0)--(10,0);
\fill (0,0) circle (1pt);
\node[above] (1) at (0,0) {$0$};
\fill (3,0) circle (1pt);
\node[above] (2) at (3,0) {$q$};
\fill (5,0) circle (1pt);
\node[above] (6) at (5,0) {$\frac{p}{2}$};
\fill (7,0) circle (1pt);
\node[above] (3) at (7,0) {$(q)$};
\fill (10,0) circle (1pt);
\node[above] (7) at (10,0) {$p$};
\end{tikzpicture}
\end{center}
Moreover, we have
$$\#A=1+\#A'<1+\frac{1}{2}(\log_2(\frac{p}{2})+\frac{3}{2})^2<\frac{1}{2}(\log_2(p)+\frac{3}{2})^2.$$

{\bf Case 2:} Suppose that $p$ is an odd number. Without loss of generality, assume that $q$ is an even number (otherwise one can consider $p-q$ instead and then obtain a $(0,p)$-mediated sequence containing $q$ through the map $x\mapsto p-x$ which clearly preserves the property of mediated sequences).

Let $q=2^kr$ for some $k,r\in\N\backslash\{0\}$ and $2\nmid r$. If $q=p-r$, then $q=\frac{q-r+p}{2}$. Since gcd$(p,q) = 1$, we have $r=1$. Let $A=\{0,\frac{1}{2}q,\frac{3}{4}q,\ldots,(1-\frac{1}{2^{k}})q,q,p\}$. For $i=1,\ldots,k$, we have $(1-\frac{1}{2^{i}})q=\frac{1}{2}(1-\frac{1}{2^{i-1}})q+\frac{1}{2}q$. Therefore, $A$ is a $(0,p)$-mediated sequence containing $q$.
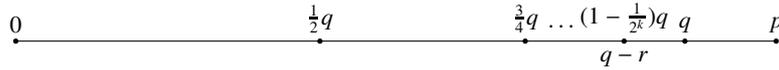
\begin{center}
{\small
\begin{tikzpicture}
\draw (0,0)--(10,0);
\fill (0,0) circle (1pt);
\node[above] (1) at (0,0) {$0$};
\fill (4,0) circle (1pt);
\node[above] (2) at (4,0) {$\frac{1}{2}q$};
\node[above] (3) at (7.2,0) {$\cdots$};
\fill (6.7,0) circle (1pt);
\node[above] (4) at (6.7,0) {$\frac{3}{4}q$};
\fill (8,0) circle (1pt);
\node[above] (5) at (8,0) {$(1-\frac{1}{2^{k}})q$};
\node[below] (8) at (8,0) {$q-r$};
\fill (8.8,0) circle (1pt);
\node[above] (6) at (8.8,0) {$q$};
\fill (10,0) circle (1pt);
\node[above] (7) at (10,0) {$p$};
\end{tikzpicture}}
\end{center}
Moreover, we have
$$\#A=k+3<\frac{1}{2}(\log_2(2^k+1)+\frac{3}{2})^2=\frac{1}{2}(\log_2(p)+\frac{3}{2})^2.$$

If $q<p-r$, then $q$ lies on the line segment between $q-r$ and $\frac{q-r+p}{2}$. Since $\frac{q-r+p}{2}-(q-r)=\frac{p+r-q}{2}<p$, then by the induction hypothesis, there exists a $(q-r,\frac{q-r+p}{2})$-mediated sequence $A'$ containing $q$ (using translations). It follows that $A=\{0,\frac{1}{2}q,\frac{3}{4}q,\ldots,(1-\frac{1}{2^{k-1}})q,p\}\cup A'$ is a $(0,p)$-mediated sequence containing $q$.
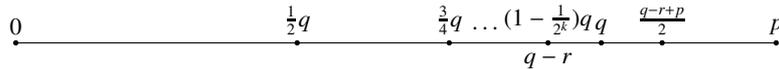
\begin{center}
{\small
\begin{tikzpicture}
\draw (0,0)--(10,0);
\fill (0,0) circle (1pt);
\node[above] (1) at (0,0) {$0$};
\fill (3.7,0) circle (1pt);
\node[above] (2) at (3.7,0) {$\frac{1}{2}q$};
\node[above] (3) at (6.2,0) {$\cdots$};
\fill (5.7,0) circle (1pt);
\node[above] (4) at (5.7,0) {$\frac{3}{4}q$};
\fill (7,0) circle (1pt);
\node[above] (5) at (7,0) {$(1-\frac{1}{2^{k}})q$};
\node[below] (8) at (7,0) {$q-r$};
\fill (7.7,0) circle (1pt);
\node[above] (6) at (7.7,0) {$q$};
\fill (8.5,0) circle (1pt);
\node[above] (8) at (8.5,0) {$\frac{q-r+p}{2}$};
\fill (10,0) circle (1pt);
\node[above] (7) at (10,0) {$p$};
\end{tikzpicture}}
\end{center}
Moreover, we have
\begin{align*}
\#A=k+1+\#A'&<\log_2(\frac{q}{r})+1+\frac{1}{2}(\log_2(\frac{p+r-q}{2})+\frac{3}{2})^2\\
&<\log_2(p)+1+\frac{1}{2}(\log_2(\frac{p}{2})+\frac{3}{2})^2\\
&=\frac{1}{2}(\log_2(p)+\frac{3}{2})^2.
\end{align*}

If $q>p-r$, then $q$ lies on the line segment between $\frac{q-r+p}{2}$ and $p$. Since $p-\frac{q-r+p}{2}=\frac{p+r-q}{2}<p$, then by the induction hypothesis, there exists a $(\frac{q-r+p}{2},p)$-mediated sequence $A'$ containing $q$ (using translations). It follows that the set $A=\{0,\frac{1}{2}q,\frac{3}{4}q,\ldots,(1-\frac{1}{2^{k}})q\}\cup A'$ is a $(0,p)$-mediated sequence containing $q$.
\begin{center}
{\small
\begin{tikzpicture}
\draw (0,0)--(10,0);
\fill (0,0) circle (1pt);
\node[above] (1) at (0,0) {$0$};
\fill (3.7,0) circle (1pt);
\node[above] (2) at (3.7,0) {$\frac{1}{2}q$};
\node[above] (3) at (6.1,0) {$\cdots$};
\fill (5.6,0) circle (1pt);
\node[above] (4) at (5.6,0) {$\frac{3}{4}q$};
\fill (6.9,0) circle (1pt);
\node[above] (5) at (6.9,0) {$(1-\frac{1}{2^{k}})q$};
\node[below] (8) at (6.9,0) {$q-r$};
\fill (7.9,0) circle (1pt);
\node[above] (6) at (7.9,0) {$\frac{q-r+p}{2}$};
\fill (8.4,0) circle (1pt);
\node[above] (8) at (8.4,0) {$q$};
\fill (10,0) circle (1pt);
\node[above] (7) at (10,0) {$p$};
\end{tikzpicture}}
\end{center}
As previously, we have $\#A =k+1+\#A' < \frac{1}{2}(\log_2(p)+\frac{3}{2})^2 $.
\end{proof}


According to the proof of Lemma \ref{sec4-lm3}, the procedure to compute an approximately minimal $(0,p)$-mediated sequence containing a given $q$ is stated as Algorithm \ref{alg:medseq}.

\begin{algorithm}
\renewcommand{\algorithmicrequire}{\textbf{Input:}}
\renewcommand{\algorithmicensure}{\textbf{Output:}}
\caption{${\tt MedSeq}(p,q)$}\label{alg:medseq}
\begin{algorithmic}[1]
\REQUIRE
$p,q\in\N,0<q<p$
\ENSURE
A set of triples $\{(u_i,v_i,w_i)\}_i$ with $u_i=\frac{1}{2}(v_i+w_i)$ such that $\{0,p\}\cup\{u_i\}_i$ is a $(0,p)$-mediated sequence containing $q$
\STATE $u:=p$, $v:=q$, $w:=$gcd$(p,q)$;
\STATE $u:=\frac{u}{w}, v:=\frac{v}{w}$;
\IF{$2|u$}
  \IF{$v=\frac{u}{2}$}
  \STATE $A:=\{(1,0,2)\}$;
  \ELSE
  \IF{$v<u/2$}
    \STATE $A:={\tt MedSeq}(\frac{u}{2},v)\cup\{(\frac{u}{2},0,u)\}$;
    \ELSE
    \STATE $A:=\{(\frac{u}{2},0,u)\}\cup({\tt MedSeq}(\frac{u}{2},v-\frac{u}{2})+\frac{u}{2})$;
  \ENDIF
  \ENDIF
\ELSE
  \IF{$2|v$}
  \STATE Let $k,r\in\N\backslash\{0\}$ such that $v=2^kr$ and $2\nmid r$;
    \IF{$v=u-r$}
    \STATE $A:=\{(\frac{1}{2}v,0,v),(\frac{3}{4}v,\frac{1}{2}v,v),\ldots,(v,v-r,u)\}$;
    \ELSE
      \IF{$v<u-r$}
      \STATE $A:=\{(\frac{1}{2}v,0,v),\ldots,(\frac{v-r+u}{2},v-r,u)\}\cup({\tt MedSeq}(\frac{u+r-v}{2},r)+v-r)$;
      \ELSE
      \STATE $A:=\{(\frac{1}{2}v,0,v),\ldots,(\frac{v-r+u}{2},v-r,u)\}\cup({\tt MedSeq}(\frac{u+r-v}{2},\frac{v+r-u}{2})+\frac{v+u-r}{2})$;
      \ENDIF
    \ENDIF
  \ELSE
  \STATE $A:=u-{\tt MedSeq}(u,u-v)$;
  \ENDIF
\ENDIF
\STATE \textbf{return} $wA$;
\end{algorithmic}
\end{algorithm}

\begin{lemma}\label{sec4-lm2}
Suppose that $\a_1$ and $\a_2$ are two rational points, and $\b$ is any rational point on the line segment between $\a_1$ and $\a_2$. Then there exists an $\{\a_1,\a_{2}\}$-rational mediated set $M$ containing $\b$. Furthermore, if the denominators of coordinates of $\a_1,\a_2,\b$ are odd numbers, and the numerators of coordinates of $\a_1,\a_2$ are even numbers, then we can ensure that the denominators of coordinates of points in $M$ are odd numbers and the numerators of coordinates of points in $M\backslash\{\b\}$ are even numbers.
\end{lemma}
\begin{proof}
Suppose $\b=(1-\frac{q}{p})\a_1+\frac{q}{p}\a_2$, $p,q\in\N,0<q<p$,gcd$(p,q) = 1$. We then construct a one-to-one correspondence between the points on the one-dimensional number axis and the points on the line across $\a_1$ and $\a_2$ via the map:
$s\mapsto(1-\frac{s}{p})\a_1+\frac{s}{p}\a_2,$
such that $\a_1$ corresponds to the origin, $\a_2$ corresponds to $p$ and $\b$ corresponds to $q$. 
Then it is clear that a $(0,p)$-mediated sequence containing $q$ corresponds to a $\{\a_1,\a_{2}\}$-rational mediated set containing $\b$.
Hence by Lemma \ref{sec4-lm3}, there exists a $\{\a_1,\a_{2}\}$-rational mediated set $M$ containing $\b$ with the number of elements less than $\frac{1}{2}(\log_2(p)+\frac{3}{2})^2$. Moreover, we can see that if $\a_1,\a_2,\b$ are lattice points, then the elements in $M$ are also lattice points.

If the denominators of coordinates of $\a_1,\a_2,\b$ are odd numbers, and the numerators of coordinates of $\a_1,\a_2$ are even numbers, assume that the least common multiple of denominators appearing in the coordinates of $\a_1,\a_2,\b$ is $r$ and then remove the denominators by multiplying the coordinates of $\a_1,\a_2,\b$ by $r$ such that $r\a_1,r\a_2$ are even lattice points. If $r\b$ is even, let $M'$ be the $\{\frac{r}{2}\a_1,\frac{r}{2}\a_{2}\}$-rational mediated set containing $\frac{r}{2}\b$ obtained as above (the elements in $M'$ are lattice points). Then $M=\frac{2}{r}M':=\{\frac{2}{r}\bu\mid\bu\in M'\}$ is an $\{\a_1,\a_{2}\}$-rational mediated set containing $\b$ such that the denominators of coordinates of points in $M$ are odd numbers and the numerators of coordinates of points in $M\backslash\{\b\}$ are even numbers.

If $r\b$ is not even, assume without loss of generality that $\b$ lies on the line segment between $\a_1$ and $\frac{\a_1+\a_2}{2}$. Let $\b'=2\b-\a_1$ with $r\b'$ an even lattice point. Let $M'$ be the $\{\frac{r}{2}\a_1,\frac{r}{2}\a_{2}\}$-rational mediated set containing $\frac{r}{2}\b'$ obtained as above (note that the elements in $M'$ are lattice points). Then $M=\frac{2}{r}M'\cup\{\b\}$ is an $\{\a_1,\a_{2}\}$-rational mediated set containing $\b$ such that the denominators of coordinates of points in $M$ are odd numbers and the numerators of coordinates of points in $M\backslash\{\b\}$ are even numbers as desired.
\end{proof}

The procedure based on Lemma \ref{sec4-lm2} to obtain an $\{\a_1,\a_{2}\}$-rational mediated set containing $\b$ is stated in Algorithm \ref{alg:lmedset}.

\begin{algorithm}
\renewcommand{\algorithmicrequire}{\textbf{Input:}}
\renewcommand{\algorithmicensure}{\textbf{Output:}}
\caption{${\tt LMedSet}(\a_1,\a_2,\b)$}\label{alg:lmedset}
\begin{algorithmic}[1]
\REQUIRE
$\a_1,\a_2,\b\in\Q^n$ such that $\b$ lies on the line segment between $\a_1$ and $\a_2$
\ENSURE
A sequence of triples $\{(\bu_i,\bv_i,\bw_i)\}_i$ with $\bu_i=\frac{1}{2}(\bv_i+\bw_i)$ such that $\{\a_1,\a_2\}\cup\{\bu_i\}_i$ is a $\{\a_1,\a_2\}$-rational mediated set containing $\b$
\STATE Let $\b=(1-\frac{q}{p})\a_1+\frac{q}{p}\a_2$, $p,q\in\N,0<q<p$,gcd$(p,q) = 1$;
\STATE $A:={\tt MedSeq}(p,q)$;
\STATE $M:=\cup_{(u,v,w)\in A}\{((1-\frac{u}{p})\a_1+\frac{u}{p}\a_2,(1-\frac{v}{p})\a_1+\frac{v}{p}\a_2,(1-\frac{w}{p})\a_1+\frac{w}{p}\a_2)\}$;
\STATE \textbf{return} $M$;
\end{algorithmic}
\end{algorithm}

\begin{lemma}\label{sec4-lm5}
For a trellis $\TT=\{\a_1,\ldots,\a_m\}$ and a lattice point $\b\in\Conv(\TT)^{\circ}$, there exists a $\TT$-rational mediated set $M_{\TT\b}$ containing $\b$ such that the denominators of coordinates of points in $M_{\TT\b}$ are odd numbers and the numerators of coordinates of points in $M_{\TT\b}\backslash\{\b\}$ are even numbers.
\end{lemma}
\begin{proof}
Suppose $\b=\sum_{i=1}^m\frac{q_i}{p}\a_i$, where $p=\sum_{i=1}^mq_i$, $p,q_i\in\N\backslash\{0\}$, gcd$(p,q_1,\ldots,q_m)=1$. If $p$ is an even number, then because gcd$(p,q_1,\ldots,q_m)=1$, there must exist an odd number among the $q_i$'s. Without loss of generality assume $q_1$ is an odd number. If $p$ is an odd number and there exists an even number among the $q_i$'s, then without loss of generality assume $q_1$ is an even number. In any of these two cases, we have
$$\b=\frac{q_1}{p}\a_1+\frac{p-q_1}{p}(\frac{q_2}{p-q_1}\a_2+\cdots+\frac{q_m}{p-q_1}\a_m).$$
Let $\b_{1}=\frac{q_2}{p-q_1}\a_2+\cdots+\frac{q_m}{p-q_1}\a_m$. Then $\b=\frac{q_1}{p}\a_1+\frac{p-q_1}{p}\b_{1}$.

If $p$ is an odd number and all $q_i$'s are odd numbers, then we have
\begin{align*}
\b&=\frac{q_1}{q_1+q_2}(\frac{q_1+q_2}{p}\a_1+\frac{q_3}{p}\a_3+\cdots+\frac{q_m}{p}\a_m)\\
&+\frac{q_2}{q_1+q_2}(\frac{q_1+q_2}{p}\a_2+\frac{q_3}{p}\a_3+\cdots+\frac{q_m}{p}\a_m).
\end{align*}
Let $\b_{1}=\frac{q_1+q_2}{p}\a_1+\frac{q_3}{p}\a_3+\cdots+\frac{q_m}{p}\a_m$ and $\b_{2}=\frac{q_1+q_2}{p}\a_2+\frac{q_3}{p}\a_3+\cdots+\frac{q_m}{p}\a_m$. Then $\b=\frac{q_1}{q_1+q_2}\b_{1}+\frac{q_2}{q_1+q_2}\b_{2}$.

Apply the same procedure for $\b_1$ (and $\b_2$), and continue iteratively. Eventually we obtain a set of points $\{\b_i\}_{i=1}^l$ such that for each $i$, $\b_i=\lambda_i\b_j+\mu_i\b_k$ or $\b_i=\lambda_i\b_j+\mu_i\a_k$ or $\b_i=\lambda_i\a_j+\mu_i\a_k$, where $\lambda_i+\mu_i=1,\lambda_i,\mu_i>0$. We claim that the denominators of coordinates of $\b_i$ are odd numbers, and the numerators of coordinates of $\b_i$ are even numbers. This is because for each $\b_i$, we have the expression $\b_i=\sum_{j}\frac{s_j}{r}\a_j$, where $r$ is an odd number and all $\a_j$'s are even lattice points. For $\b_i=\lambda\b_j+\mu\b_k$ (or $\b_i=\lambda\b_j+\mu\a_k$, $\b_i=\lambda\a_j+\mu\a_k$ respectively), let $M_i$ be the $\{\b_{j},\b_{k}\}$- (or $\{\b_{j},\a_{k}\}$-, $\{\a_{j},\a_{k}\}$- respectively) rational mediated set containing $\b_i$ obtained by Lemma \ref{sec4-lm2} such that the denominators of coordinates of points in $M_i$ are odd numbers and the numerators of coordinates of points in $M_i\backslash\{\b\}$ are even numbers for $i=0,\ldots,l$ (set $\b_0=\b$). Let $M_{\TT\b}=\cup_{i=0}^lM_i$. Then $M_{\TT\b}$ is clearly a $\TT$-rational mediated set containing $\b$ with the desired property.
\end{proof}

\subsection{Decomposing SONC with binomial squares}
\label{sec:soncbinomial}
For $r\in\N$ and $f(\x)\in\R[\x]$, let $f(\x^r):=f(x_1^r,\ldots,x_n^r)$. For an odd $r\in\N$,  it is clear that $f(\x)=\sum_{\a\in\TT}c_{\a}\x^{\a}-d\x^{\b}$ is a nonnegative circuit if and only if  $f(\x^r)=\sum_{\a\in\TT}c_{\a}\x^{r\a}-d\x^{r\b}$ is a nonnegative circuit.
\begin{theorem}\label{sec3-thm3}
Let $f=\sum_{\a\in\TT} c_{\a}\x^{\a}-d\x^{\b}\in\R[\x]$ with $d\ne0$ be a circuit polynomial. Assume that $M_{\TT\b}$ is a $\TT$-rational mediated set containing $\b$ provided by Lemma \ref{sec4-lm5}.
For each $\bu\in M_{\TT\b}\backslash\TT$, let $\bu=\frac{1}{2}(\bv_{\bu}+\bw_{\bu}), \bv_{\bu}\ne\bw_{\bu}\in M_{\TT\b}$. Then $f$ is nonnegative if and only if $f$ can be written as $f=\sum_{\bu\in M_{\TT\b}\backslash\TT}(a_{\bu}\x^{\frac{1}{2}\bv_{\bu}}-b_{\bu}\x^{\frac{1}{2}\bw_{\bu}})^2+c\x^{\a}$, $a_{\bu},b_{\bu}\in\R,c\ge0,\a\in\TT$.
\end{theorem}
\begin{proof}
Assume that the least common multiple of denominators appearing in the coordinates of points in $M_{\TT\b}$ is $r$, which is odd. Then $f(\x)$ is nonnegative if and only if $f(\x^r)$ is nonnegative. Multiply all coordinates of points in $M_{\TT\b}$ by $r$ to remove the denominators, and the obtained $rM_{\TT\b}:=\{r\bu\mid\bu\in M_{\TT\b}\}$ is an $r\TT$-mediated set containing $r\b$. Hence by Theorem \ref{sec3-thm1}, $f(\x^r)$ is nonnegative if and only if $f(\x^r)$ can be written as $f(\x^r)=\sum_{\bu\in M_{\TT\b}\backslash \TT}(a_{\bu}\x^{\frac{r}{2}\bv_{\bu}}-b_{\bu}\x^{\frac{r}{2}\bw_{\bu}})^2+c\x^{r\a}$, $a_{\bu},b_{\bu}\in\R,c\ge0,\a\in\TT$, which is equivalent to $f(\x)=\sum_{\bu\in M_{\TT\b}\backslash\TT}(a_{\bu}\x^{\frac{1}{2}\bv_{\bu}}-b_{\bu}\x^{\frac{1}{2}\bw_{\bu}})^2+c\x^{\a}$.
\end{proof}

\begin{example}
Let $f=x^4y^2+x^2y^4+1-3x^2y^2$ be the Motzkin polynomial and $\TT=\{\a_1=(0,0),\a_2=(4,2),\a_3=(2,4)\}$, $\b=(2,2)$. 
Let $\b_1=\frac{1}{3}\a_1+\frac{2}{3}\a_2$ and $\b_2=\frac{1}{3}\a_1+\frac{2}{3}\a_3$ such that $\b=\frac{1}{2}\b_1+\frac{1}{2}\b_2$. Let $\b_3=\frac{2}{3}\a_1+\frac{1}{3}\a_2$ and $\b_4=\frac{2}{3}\a_1+\frac{1}{3}\a_3$. Then  $M=\{\a_1,\a_2,\a_3,\b,\b_1,\b_2,\b_3,\b_4\}$ is a $\TT$-rational mediated set containing $\b$.

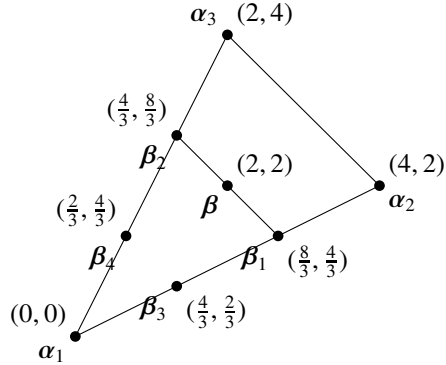
\begin{center}
\begin{tikzpicture}
\draw (0,0)--(2,4);
\draw (0,0)--(4,2);
\draw (2,4)--(4,2);
\draw (4/3,8/3)--(8/3,4/3);
\fill (0,0) circle (2pt);
\node[above left] (1) at (0,0) {$(0,0)$};
\node[below left] (1) at (0,0) {$\a_1$};
\fill (2,4) circle (2pt);
\node[above right] (2) at (2,4) {$(2,4)$};
\node[above left] (2) at (2,4) {$\a_3$};
\fill (4,2) circle (2pt);
\node[above right] (3) at (4,2) {$(4,2)$};
\node[below right] (3) at (4,2) {$\a_2$};
\fill (2,2) circle (2pt);
\node[above right] (4) at (2,2) {$(2,2)$};
\node[below left] (4) at (2,2) {$\b$};
\fill (4/3,8/3) circle (2pt);
\node[above left] (5) at (4/3,8/3) {$(\frac{4}{3},\frac{8}{3})$};
\node[below left] (5) at (4/3,8/3) {$\b_2$};
\fill (8/3,4/3) circle (2pt);
\node[below right] (6) at (8/3,4/3) {$(\frac{8}{3},\frac{4}{3})$};
\node[below left] (6) at (8/3,4/3) {$\b_1$};
\fill (2/3,4/3) circle (2pt);
\node[above left] (5) at (2/3,4/3) {$(\frac{2}{3},\frac{4}{3})$};
\node[below left] (5) at (2/3,4/3) {$\b_4$};
\fill (4/3,2/3) circle (2pt);
\node[below right] (6) at (4/3,2/3) {$(\frac{4}{3},\frac{2}{3})$};
\node[below left] (6) at (4/3,2/3) {$\b_3$};
\end{tikzpicture}
\end{center}

By Theorem \ref{sec3-thm3}, one has  $f=x^4y^2+x^2y^4+1-3x^2y^2=(a_1x^{\frac{2}{3}}y^{\frac{4}{3}}-b_1x^{\frac{4}{3}}y^{\frac{2}{3}})^2+(a_2xy^2-b_2x^{\frac{1}{3}}y^{\frac{2}{3}})^2+(a_3x^{\frac{2}{3}}y^{\frac{4}{3}}-b_3)^2
+(a_4x^2y-b_4x^{\frac{2}{3}}y^{\frac{1}{3}})^2+(a_5x^{\frac{4}{3}}y^{\frac{2}{3}}-b_5)^2$.
Comparing coefficients yields  $f=\frac{3}{2}(x^{\frac{2}{3}}y^{\frac{4}{3}}-x^{\frac{4}{3}}y^{\frac{2}{3}})^2+(xy^2-x^{\frac{1}{3}}y^{\frac{2}{3}})^2+\frac{1}{2}(x^{\frac{2}{3}}y^{\frac{4}{3}}-1)^2
+(x^2y-x^{\frac{2}{3}}y^{\frac{1}{3}})^2+\frac{1}{2}(x^{\frac{4}{3}}y^{\frac{2}{3}}-1)^2$, a sum of five binomial squares with rational exponents.
\end{example}

\begin{lemma}\label{sec3-lm6}
Let $f(\x)\in\R[\x]$. For an odd number $r$, $f(\x)\in\SONC$ if and only if $f(\x^r)\in\SONC$.
\end{lemma}
\begin{proof}
It comes from the fact that $f(\x)$ is a nonnegative circuit if and only if $f(\x^r)$ is a nonnegative circuit  for an odd number $r$.
\end{proof}

\begin{theorem}\label{sec3-thm5}
Let $f=\sum_{\a\in\Lambda(f)}c_{\a}\x^{\a}-\sum_{\b\in\Gamma(f)}d_{\b}\x^{\b}\in\R[\x]$. Let $\CC(\b)$ be as in \eqref{fb}. For every $\b\in\Gamma(f)$ and every trellis $\TT\in\CC(\b)$, let $M_{\TT\b}$ be a $\TT$-rational mediated set containing $\b$ provided by Lemma \ref{sec4-lm5}.
Let $M=\cup_{\b\in\Gamma(f)}\cup_{\TT\in\CC(\b)}M_{\TT\b}$. For each $\bu\in M\backslash\Lambda(f)$, let $\bu=\frac{1}{2}(\bv_{\bu}+\bw_{\bu}),\bv_{\bu}\ne\bw_{\bu}\in M$. Let $\tilde{\A}=\{\a\in\Lambda(f)\mid\a\notin\cup_{\b\in\Gamma(f)}\cup_{\TT\in\CC(\b)}\TT\}$. Then $f\in\SONC$ if and only if $f$ can be written as $f=\sum_{\bu\in M\backslash\Lambda(f)}(a_{\bu}\x^{\frac{1}{2}\bv_{\bu}}-b_{\bu}\x^{\frac{1}{2}\bw_{\bu}})^2+\sum_{\b\in\Gamma(f)}c'_{\a_{\b}}\x^{\a_{\b}}+\sum_{\a\in \tilde{\A}}c_{\a}\x^{\a}$, $a_{\bu},b_{\bu}\in\R,c'_{\a_{\b}}\ge0,\a_{\b}\in\CC(\b),\b\in\Gamma(f)$.
\end{theorem}
\begin{proof}
Suppose $f\in\SONC$. By Theorem 5.5 in \cite{wang}, we can write $f$ as
\[
f=\sum_{\b\in\Gamma(f)}\sum_{\TT\in\CC(\b)}f_{\TT\b}+\sum_{\a\in \tilde{\A}}c_{\a}\x^{\a} \,,
\]
such that every $f_{\TT\b}$ is a nonnegative circuit supported on $\TT\cup\{\b\}$. 
We have 
\[
f_{\TT\b}=\sum_{\bu\in M_{\TT\b}\backslash\TT}(a_{\bu}\x^{\frac{1}{2}\bv_{\bu}}-b_{\bu}\x^{\frac{1}{2}\bw_{\bu}})^2+c'_{\a_{\b}}\x^{\a_{\b}} \,,
\] 
for some $a_{\bu},b_{\bu}\in\R,c'_{\a_{\b}}\ge0,\a_{\b}\in\CC(\b)$ by Theorem \ref{sec3-thm3}. 
Thus there exist $a_{\bu},b_{\bu}\in\R,c'_{\a_{\b}}\ge0,\a_{\b}\in\CC(\b),\b\in\Gamma(f)$ such that 
\[
f=\sum_{\bu\in M\backslash\Lambda(f)}(a_{\bu}\x^{\frac{1}{2}\bv_{\bu}}-b_{\bu}\x^{\frac{1}{2}\bw_{\bu}})^2+\sum_{\b\in\Gamma(f)}c'_{\a_{\b}}\x^{\a_{\b}}+\sum_{\a\in \tilde{\A}}c_{\a}\x^{\a} \,.
\]
Now suppose that $f$ has the desired form. 
Assume that the least common multiple of denominators appearing in the coordinates of points in $M$ is an odd positive integer $r$. 
Then there exist $a_{\bu},b_{\bu}\in\R,c'_{\a_{\b}}\ge0,\a_{\b}\in\CC(\b),\b\in\Gamma(f)$ such that
\[
f(\x^r)=\sum_{\bu\in M\backslash\Lambda(f)}(a_{\bu}\x^{\frac{r}{2}\bv_{\bu}}-b_{\bu}\x^{\frac{r}{2}\bw_{\bu}})^2+\sum_{\b\in\Gamma(f)}c'_{\a_{\b}}\x^{r\a_{\b}}+\sum_{\a\in \tilde{\A}}c_{\a}\x^{r\a} \,.
\]
Thus $f(\x^r)\in\SONC$ since every binomial or monomial square is a nonnegative circuit. 
Hence by Lemma \ref{sec3-lm6}, $f(\x)\in\SONC$.
\end{proof}

\section{SOC representations of SONC cones}
\label{sec:socrepr}

SOCP plays an important role in convex optimization and can be handled via very efficient algorithms. If a SOC representation exists for a given convex cone, then it is possible to design efficient algorithms for optimization problems over the convex cone. In \cite{fa}, Fawzi proved that PSD cones do not admit any SOC representations in general, which implies that SOS cones do not admit any SOC representations in general. In this section, we prove that dramatically unlike the SOS cones, SONC cones always admit SOC representations. Let us discuss it in more details. 
Let $\mathcal{Q}^k:=\mathcal{Q}\times\cdots \times \mathcal{Q}$ be the Cartesian product of $k$ copies of a SOC $\mathcal{Q}$. A {\em linear slice} of $\mathcal{Q}^k$ is an intersection of $\mathcal{Q}^k$ with a linear subspace.
\begin{definition}
A convex cone $C\subseteq\R^m$ has a {\em SOC lift of size $k$} (or simply a {\em $\mathcal{Q}^k$-lift}) if it can be written as the projection of a slice of $\mathcal{Q}^k$, that is, there is a subspace $L$ of $\mathcal{Q}^k$ and a linear map $\pi\colon\mathcal{Q}^k\rightarrow\R^m$ such that $C=\pi(\mathcal{Q}^k\cap L)$.
\end{definition}

We give the following definition of SONC cones supported on given lattice points.
\begin{definition}\label{def-cone}
Given sets of lattice points $\A\subseteq(2\N)^n$, $\B_1\subseteq\Conv(\A)\cap(2\N)^n$ and $\B_2\subseteq\Conv(\A)\cap(\N^n\backslash(2\N)^n)$ such that $\A\cap\B_1=\varnothing$, define the SONC cone supported on $\A,\B_1,\B_2$ as
\begin{align*}
\SONC_{\A,\B_1,\B_2}:=& \{(\mathbf{c}_{\A},\mathbf{d}_{\B_1},\mathbf{d}_{\B_2})\in\R_+^{|\A|}\times\R_+^{|\B_1|}\times\R^{|\B_2|}\\
&\mid\sum_{\a\in\A} c_{\a}\x^{\a}-\sum_{\b\in\B_1\cup\B_2}d_{\b}\x^{\b}\in\SONC \},
\end{align*}
where $\mathbf{c}_{\A}=(c_{\a})_{\a\in\A}$, $\mathbf{d}_{\B_1}=(d_{\b})_{\b\in\B_1}$ and $\mathbf{d}_{\B_2}=(d_{\b})_{\b\in\B_2}$. It is easy to check that $\SONC_{\A,\B_1,\B_2}$ is indeed a convex cone.
\end{definition}

Let $\mathbb{S}^2_+$ be the convex cone of $2\times2$ positive semidefinite matrices
\[\mathbb{S}^2_+:=\left\{\begin{bmatrix}a&b\\b&c\end{bmatrix}\in\R^{2\times2}\mid\begin{bmatrix}a&b\\b&c\end{bmatrix}\textrm{ is positive semidefinite}\right\}.\]
\begin{lemma}
$\mathbb{S}^2_+$ is a $3$-dimensional rotated SOC.
\end{lemma}
\begin{proof}
The condition for $\begin{bmatrix}a&b\\b&c\end{bmatrix}$ to be a positive semidefinite matrix is $a\ge0,c\ge0,ac\ge b^2$. 
Thus $\mathbb{S}^2_+$ is a rotated SOC by the definition.
\end{proof}

Now we are ready to prove the main theorem.
\begin{theorem}\label{socr}
For $\A\subseteq(2\N)^n$, $\B_1\subseteq\Conv(\A)\cap(2\N)^n$ and $\B_2\subseteq\Conv(\A)\cap(\N^n\backslash(2\N)^n)$ such that $\A\cap\B_1=\varnothing$, the convex cone $\SONC_{\A,\B_1,\B_2}$ has an $(\mathbb{S}^2_+)^k$-lift for some $k\in\N$.
\end{theorem}
\begin{proof}
For every $\b\in\B_1\cup\B_2$, let $\CC(\b)$ be as in \eqref{fb}. Then
for every $\b\in\B_1\cup\B_2$ and every $\TT\in\CC(\b)$, let $M_{\TT\b}$ be the $\TT$-rational mediated set containing $\b$ provided by Lemma \ref{sec4-lm5}.
Let $M=\cup_{\b\in\B_1\cup\B_2}\cup_{\TT\in\CC(\b)}M_{\TT\b}$. 
For each $\bu_i\in M\backslash\A$, let us write  $\bu_i=\frac{1}{2}(\bv_i+\bw_i)$. Let $\tilde{\A}=\{\a\in\A\mid\a\notin\cup_{\b\in\B_1\cup\B_2}\cup_{\TT\in\CC(\b)}\TT\}$ and $k=\#M\backslash\A+\#\tilde{\A}$.

Then by Theorem \ref{sec3-thm5}, a polynomial $f$ is in $\SONC_{\A,\B_1,\B_2}$ if and only if $f$ can be written as $f=\sum_{\bu_i\in M\backslash\A}(a_{i}\x^{\frac{1}{2}\bv_i}-b_{i}\x^{\frac{1}{2}\bw_i})^2+\sum_{\b\in\Gamma(f)}c'_{\a_{\b}}\x^{\a_{\b}}+\sum_{\a\in \tilde{\A}}c_{\a}\x^{\a}$, $a_i,b_i\in\R,c'_{\a_{\b}}\ge0,\a_{\b}\in\CC(\b),\b\in\Gamma(f)$, which is equivalent to the existence of matrices $Q_1, \cdots, Q_k\in\mathbb{S}^2_+$ such that 
\begin{equation}\label{sobs}
f=\sum_{\bu_i\in M\backslash\A}\left[\frac{1}{2}\bv_i,\frac{1}{2}\bw_i\right]Q_i\left[\frac{1}{2}\bv_i,\frac{1}{2}\bw_i\right]^T+\sum_{\a\in \tilde{\A}\cup\{\a_{\b}\}_{\b\in\Gamma(f)},Q_i}\left[\frac{1}{2}\a,\mathbf{0}\right]Q_i\left[\frac{1}{2}\a,\mathbf{0}\right]^T.
\end{equation}

Let $\pi:(\mathbb{S}^2_+)^{k}\rightarrow\SONC_{\A,\B_1,\B_2}$
be the linear map that maps an element $(Q_1, \cdots, Q_k)$ in $(\mathbb{S}^2_+)^{k}$ to the coefficient vector of $f$ which is in $\SONC_{\A,\B_1,\B_2}$ via the equality \eqref{sobs}. Hence $\pi$ yields an $(\mathbb{S}^2_+)^k$-lift for $\SONC_{\A,\B_1,\B_2}$.
\end{proof}



\revision{{\bf Comparison with Averkov's construction on the size of SOC lifts.} Given a SONC cone, the SOC lift due to \cite{abe} is built on SOC lifts for hypographs of weighted geometric means. In his paper, Averkov appealed to the construction given in \cite{ben2001lectures} for SOC lifts of weighted geometric means, which is of size $\bigo(p)$ where $p$ is the least common denominator of the weights (i.e., the barycentric coordinates $\{\lambda_i\}$). 
It is possible to produce such lifts of smaller size. Indeed, \cite{sagnol2013} presented a construction on matrix geometric means, which yields a SOC lift of size $2\lfloor\log_2(p)\rfloor+1$ for bivariate weighted geometric means with $p$ being the denominator of the weight; see also \cite{fawzi2017lieb}. 
Our construction of mediated sequences also equivalently yields a SOC lift for bivariate weighted geometric means. 
Even though we can only prove a quadratic bound in terms of $\log_2(p)$ for the size of this construction (Lemma \ref{sec4-lm3}), we empirically found that it actually depends linearly on $\log_2(p)$, which is approximately $1.29\log_2(p)-0.26$. The following table provides a comparison on the sizes of different constructions for (hypographs of) SOC lifts of bivariate weighted geometric means for $p=10,10^2,\ldots,10^8$, where the average size of our construction is taken over all integers $q$ such that the weight $t=q/p$, $0<q<p$ and gcd$(p,q)=1$.

\begin{center}
\begin{tabular}{|c|c|c|c|c|c|c|c|c|}
\hline
the denominator $p$&$10$&$10^2$&$10^3$&$10^4$&$10^5$&$10^6$&$10^7$&$10^8$\\
\hline
the size predicted by Conjecture \ref{con}&4&7&10&14&17&20&24&27\\
\hline
the average size of our construction&4.0&8.4&12.5&16.8&21.2&25.4&29.7&34.0\\
\hline
the size of Sagnol's construction&7&13&19&27&33&39&47&53\\
\hline
\end{tabular}
\end{center}
}

\section{SONC optimization via SOCP}
\label{sec:optim}
In this section, we tackle the following unconstrained polynomial optimization problem via SOCP, based on the representation of SONC cones derived in the previous section:
%
\begin{equation}\label{sec3-eq00}
(\textrm{P}):\quad
 \sup \{\xi : f(\x)-\xi\ge0,\quad\x\in\R^n
\} \,.
\end{equation}
\if{
\begin{equation}\label{sec3-eq00}
(\textrm{P}):\quad\begin{cases}
\sup&\xi\\
\textrm{s.t.}&f(\x)-\xi\ge0,\quad\x\in\R^n.
\end{cases}
\end{equation}
}\fi
Let us  denote by $\xi^*$ the optimal value of (\ref{sec3-eq00}).
Replace the nonnegativity constraint in (\ref{sec3-eq00}) by the following one to obtain a SONC relaxation with optimal value denoted by $\xi_{sonc}$:
\begin{equation}\label{sec3-eq5}
(\textrm{SONC}):\quad
\sup \{\xi : f(\x)-\xi\in\SONC \} \,.
\end{equation}
\if{
\begin{equation}\label{sec3-eq5}
(\textrm{SONC}):\quad\begin{cases}
\sup&\xi\\
\textrm{s.t.}&f(\x)-\xi\in\SONC.
\end{cases}
\end{equation}
}\fi

\subsection{Conversion to PN-polynomials}
Suppose $f=\sum_{\a\in\Lambda(f)}c_{\a}\x^{\a}-\sum_{\b\in\Gamma(f)}d_{\b}\x^{\b}\in\R[\x]$. If $d_{\b}>0$ for all $\b\in\Gamma(f)$, then we call $f$ a {\em PN-polynomial}. The prefix
``PN" in PN-polynomial is short for ``positive part plus negative part".
The positive part is given by $\sum_{\a\in\Lambda(f)}c_{\a}\x^{\a}$ and the negative part is given by $-\sum_{\b\in\Gamma(f)}d_{\b}\x^{\b}$. 
For a PN-polynomial $f(\x)$, it is clear that
$f(\x)\ge0\textrm{ for all }\x\in\R^n$ if and only if  $f(\x)\ge0\textrm{ for all }\x\in\R_+^n$.
Moreover, we have the following lemma.
\begin{lemma}\label{sec4-lm}
Let $f(\x)\in\R[\x]$ be a PN-polynomial. Then for any positive integer $k$, $f(\x)\in\SONC$ if and only if $f(\x^k)\in\SONC$.
\end{lemma}
\begin{proof}
It comes from the fact that a polynomial $f(\x)$ with exactly one negative term is a nonnegative circuit if and only if $f(\x^k)$ is a nonnegative circuit for any positive integer $k$.
\end{proof}

\begin{theorem}
\label{sec4-thm}
Let $f=\sum_{\a\in\Lambda(f)}c_{\a}\x^{\a}-\sum_{\b\in\Gamma(f)}d_{\b}\x^{\b}\in\R[\x]$ be a PN-polynomial. Let $\CC(\b)$ be as in \eqref{fb}. For every $\b\in\Gamma(f)$ and every $\TT\in\CC(\b)$, let $M_{\TT\b}$ be a $\TT$-rational mediated set containing $\b$. 
Let $M=\cup_{\b\in\Gamma(f)}\cup_{\TT\in\CC(\b)}M_{\TT\b}$ and $\tilde{\A}=\{\a\in\Lambda(f)\mid\a\notin\cup_{\b\in\Gamma(f)}\cup_{\TT\in\CC(\b)}\TT\}$. For each $\bu\in M\backslash\Lambda(f)$, let $\bu=\frac{1}{2}(\bv+\bw)$. 
Then $f\in\SONC$ if and only if $f$ can be written as $f=\sum_{\bu\in M\backslash\Lambda(f)}(a_{\bu}\x^{\frac{1}{2}\bv}-b_{\bu}\x^{\frac{1}{2}\bw})^2+\sum_{\b\in\Gamma(f)}c'_{\a_{\b}}\x^{\a_{\b}}+\sum_{\a\in \tilde{\A}}c_{\a}\x^{\a}$, $a_{\bu},b_{\bu}\in\R,c'_{\a_{\b}}\ge0,\a_{\b}\in\CC(\b),\b\in\Gamma(f)$.
\end{theorem}
\begin{proof}
It follows easily from Lemma \ref{sec4-lm} and Theorem \ref{sec3-thm1}.
\end{proof}

The significant difference between Theorem \ref{sec3-thm5} and Theorem \ref{sec4-thm} is that to represent a SONC PN-polynomial as a sum of binomial squares, we do not require the denominators of coordinates of points in $\TT$-rational mediated sets to be odd. 
By virtue of this fact, for given trellis $\TT=\{\a_1,\ldots,\a_m\}$ and lattice point $\b\in\Conv(\TT)^{\circ}$, we can then construct a $\TT$-rational mediated set $M_{\TT\b}$ containing $\b$ which is smaller than that the one from Lemma \ref{sec4-lm5}.
\begin{lemma}\label{sec4-lm4}
For a trellis $\TT$ and a lattice point $\b\in\Conv(\TT)^{\circ}$, there is a $\TT$-rational mediated set $M_{\TT\b}$ containing $\b$.
\end{lemma}
\begin{proof}
Suppose that $\b=\sum_{i=1}^m\frac{q_i}{p}\a_i$, where $p=\sum_{i=1}^mq_i$, $p,q_i\in\N\backslash\{0\}$, gcd$(p,q_1,\ldots,q_m)=1$. We can write $$\b=\frac{q_1}{p}\a_1+\frac{p-q_1}{p}(\frac{q_2}{p-q_1}\a_2+\cdots+\frac{q_m}{p-q_1}\a_m).$$
Let $\b_{1}=\frac{q_2}{p-q_1}\a_2+\cdots+\frac{q_m}{p-q_1}\a_m$. Then $\b=\frac{q_1}{p}\a_1+\frac{p-q_1}{p}\b_{1}$. Apply the same procedure for $\b_1$, and continue like this. Eventually we obtain a set of points $\{\b_i\}_{i=0}^{m-2}$ (set $\b_0=\b$) such that $\b_i=\lambda_i\a_{i+1}+\mu_i\b_{i+1}$, $i=0,\ldots,m-3$ and $\b_{m-2}=\lambda_{m-2}\a_{m-1}+\mu_{m-2}\a_m$, where $\lambda_i+\mu_i=1,\lambda_i,\mu_i>0$, $i=0,\ldots,m-2$. For $\b_i=\lambda_i\a_{i+1}+\mu_i\b_{i+1}$ (resp. $\b_{m-2}=\lambda_{m-2}\a_{m-1}+\mu_{m-2}\a_m$), let $M_i$ be the $\{\a_{i+1},\b_{i+1}\}$- (resp. $\{\a_{m-1},\a_{m}\}$-) rational mediated set containing $\b_i$ obtained by Lemma \ref{sec4-lm2}, $i=0,\ldots,m-2$. Let $M_{\TT\b}=\cup_{i=0}^{m-2}M_i$. Then clearly $M_{\TT\b}$ is a $\TT$-rational mediated set containing $\b$.
\end{proof}

The procedure based on Lemma \ref{sec4-lm4} to obtain a $\TT$-rational mediated sets is stated in Algorithm \ref{alg:medset}.

\begin{algorithm}
\renewcommand{\algorithmicrequire}{\textbf{Input:}}
\renewcommand{\algorithmicensure}{\textbf{Output:}}
\caption{${\tt MedSet}(\TT,\b)$}\label{alg:medset}
\begin{algorithmic}[1]
\REQUIRE
A trellis $\TT=\{\a_1,\ldots,\a_m\}$ and a lattice point $\b\in\Conv(\TT)^{\circ}$
\ENSURE
A set of triples $\{(\bu_i,\bv_i,\bw_i)\}_i$ with $\bu_i=\frac{1}{2}(\bv_i+\bw_i)$ such that $\TT\cup\{\bu_i\}_i$ is a $\TT$-rational mediated set containing $\b$
\STATE Let $\b=\sum_{i=1}^m\frac{q_i}{p}\a_i$, where $p=\sum_{i=1}^mq_i$, $p,q_i\in\N\backslash\{0\}$, gcd$(p,q_1,\ldots,q_m)=1$;
\STATE $k:=1, \b_0:=\b$;
\WHILE{$k<m-1$}
\STATE $\b_k:=\frac{q_{k+1}}{p-(q_1+\cdots+q_k)}\a_{k+1}+\cdots+\frac{q_{m}}{p-(q_1+\cdots+q_k)}\a_{m}$;
\STATE $M_{k-1}:={\tt LMedSeq}(\a_{k},\b_{k},\b_{k-1})$;
\ENDWHILE
\STATE $M_{m-2}:={\tt LMedSeq}(\a_{m-1},\a_{m},\b_{m-2})$;
\STATE $M:=\cup_{i=0}^{m-2}M_i$;
\STATE \textbf{return} $M$;
\end{algorithmic}
\end{algorithm}

\begin{example}\label{ex2}
Let $f=x^4y^2+x^2y^4+1-3x^2y^2$ be the Motzkin polynomial and $\TT=\{\a_1=(4,2),\a_2=(2,4),\a_3=(0,0)\}$, $\b=(2,2)$. Then $\b=\frac{1}{3}\a_1+\frac{1}{3}\a_2+\frac{1}{3}\a_3=\frac{1}{3}\a_1+\frac{2}{3}(\frac{1}{2}\a_2+\frac{1}{2}\a_3)$. Let $\b_1=\frac{1}{2}\a_2+\frac{1}{2}\a_3$ such that  $\b=\frac{1}{3}\a_1+\frac{2}{3}\b_1$. Let $\b_2=\frac{2}{3}\a_1+\frac{1}{3}\b_1$. Then it is easy to check that $M=\{\a_1,\a_2,\a_3,\b,\b_1,\b_2\}$ is a $\TT$-rational mediated set containing $\b$.
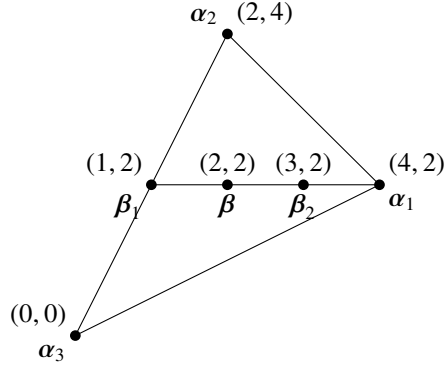
\begin{center}
\begin{tikzpicture}
\draw (0,0)--(2,4);
\draw (0,0)--(4,2);
\draw (2,4)--(4,2);
\draw (1,2)--(4,2);
\fill (0,0) circle (2pt);
\node[above left] (1) at (0,0) {$(0,0)$};
\node[below left] (1) at (0,0) {$\a_3$};
\fill (2,4) circle (2pt);
\node[above right] (2) at (2,4) {$(2,4)$};
\node[above left] (2) at (2,4) {$\a_2$};
\fill (4,2) circle (2pt);
\node[above right] (3) at (4,2) {$(4,2)$};
\node[below right] (3) at (4,2) {$\a_1$};
\fill (2,2) circle (2pt);
\node[above] (4) at (2,2) {$(2,2)$};
\node[below] (4) at (2,2) {$\b$};
\fill (1,2) circle (2pt);
\node[above left] (5) at (1,2) {$(1,2)$};
\node[below left] (5) at (1,2) {$\b_1$};
\fill (3,2) circle (2pt);
\node[above] (6) at (3,2) {$(3,2)$};
\node[below] (6) at (3,2) {$\b_2$};
\end{tikzpicture}
\end{center}
By a simple computation, we have $f=(1-xy^2)^2+2(x^{\frac{1}{2}}y-x^{\frac{3}{2}}y)^2+(xy-x^2y)^2$. Here we represent $f$ as a sum of three binomial squares with rational exponents.
\end{example}

We associate to a polynomial $f=\sum_{\a\in\Lambda(f)}c_{\a}\x^{\a}-\sum_{\b\in\Gamma(f)}d_{\b}\x^{\b}$, the PN-polynomial $\tilde{f}=\sum_{\a\in\Lambda(f)}c_{\a}\x^{\a}-\sum_{\b\in\Gamma(f)}|d_{\b}|\x^{\b}$.
\begin{lemma}\label{sec4-lm1}
Suppose $f=\sum_{\a\in\Lambda(f)}c_{\a}\x^{\a}-\sum_{\b\in\Gamma(f)}d_{\b}\x^{\b}\in\R[\x]$. If $\tilde{f}$ is nonnegative, then $f$ is nonnegative. Moreover, $\tilde{f}\in\SONC$ if and only if $f\in\SONC$.
\end{lemma}
\begin{proof}
For any $\x\in\R^n$, we have
\begin{align*}
f(\x)&=\sum_{\a\in\Lambda(f)}c_{\a}\x^{\a}-\sum_{\b\in\Gamma(f)}d_{\b}\x^{\b}\\
&\ge\sum_{\a\in\Lambda(f)}c_{\a}|\x|^{\a}-\sum_{\b\in\Gamma(f)}|d_{\b}||\x|^{\b} =\tilde{f}(|\x|),
\end{align*}
where $|\x|=(|x_1|,\ldots,|x_n|)$. It follows that the nonnegativity of $\tilde{f}$ implies the nonnegativity of $f$.

For every $\b\in\Gamma(f)$, let $\CC(\b)$ be as in \eqref{fb}. Let $\B=\{\b\in\Gamma(f)\mid\b\notin(2\N)^n\textrm{ and }d_{\b}<0\}$ and  $\tilde{\A}=\{\a\in\Lambda(f)\mid\a\notin\cup_{\b\in\Gamma(f)}\cup_{\TT\in\CC(\b)}\TT\}$. Assume $\tilde{f}\in\SONC$. Then we can write
\begin{align*}
\tilde{f}=&\sum_{\b\in\Gamma(f)\backslash\B}\sum_{\TT\in\CC(\b)}\left(\sum_{\a\in \TT}c_{\TT\b\a}\x^{\a}-d_{\TT\b}\x^{\b}\right)\\
&+\sum_{\b\in\B}\sum_{\TT\in\CC(\b)}\left(\sum_{\a\in \TT}c_{\TT\b\a}\x^{\a}-\tilde{d}_{\TT\b}\x^{\b}\right)+\sum_{\a\in \tilde{\A}}c_{\a}\x^{\a}
\end{align*}
such that each $\sum_{\a\in \TT}c_{\TT\b\a}\x^{\a}-d_{\TT\b}\x^{\b}$ and each $\sum_{\a\in \TT}c_{\TT\b\a}\x^{\a}-\tilde{d}_{\TT\b}\x^{\b}$ are nonnegative circuits. Note that $\sum_{\a\in \TT}c_{\TT\b\a}$ $\x^{\a}+\tilde{d}_{\TT\b}\x^{\b}$ is also a nonnegative circuit and $\sum_{\TT\in\TT(\b)}\tilde{d}_{\TT\b}$ $=|d_{\b}|=-d_{\b}$ for any $\b\in\B$. Hence,
\begin{align*}
f=&\sum_{\b\in\Gamma(f)\backslash\B}\sum_{\TT\in\CC(\b)}\left(\sum_{\a\in \TT}c_{\TT\b\a}\x^{\a}-d_{\TT\b}\x^{\b}\right)\\
&+\sum_{\b\in\B}\sum_{\TT\in\CC(\b)}\left(\sum_{\a\in \TT}c_{\TT\b\a}\x^{\a}+\tilde{d}_{\TT\b}\x^{\b}\right)+\sum_{\a\in \tilde{\A}}c_{\a}\x^{\a}\in\SONC.
\end{align*}

The inverse follows similarly.
\end{proof}

Hence by Lemma \ref{sec4-lm1}, if we replace the polynomial $f$ in (\ref{sec3-eq5}) by its associated PN-polynomial $\tilde{f}$, then this does not affect the optimal value of (\ref{sec3-eq5}):
\begin{equation}\label{sec3-eq6}
(\textrm{SONC-PN}):\quad
\begin{cases}
\sup&\xi\\
\textrm{s.t.}&\tilde{f}(\x)-\xi\in\SONC\,.
\end{cases}
\end{equation}

\begin{remark}
Lemma \ref{sec4-lm1} tells us that the SONC formulation for the polynomial optimization problem \eqref{sec3-eq5} always provides the optimal value w.r.t. the corresponding PN-polynomial. See also Remark \ref{quality}.
\end{remark}

\subsection{Compute a simplex cover}
Given a polynomial $f=\sum_{\a\in\Lambda(f)}c_{\a}\x^{\a}-\sum_{\b\in\Gamma(f)}d_{\b}\x^{\b}\in\R[\x]$, in order to obtain a SONC decomposition of $f$, we use all simplices covering $\b$ for each $\b\in\Gamma(f)$ in Theorem \ref{sec3-thm5}. 
In practice, we do not need that many simplices, as illustrated by the following example.
\begin{example}
Let $f=50x^4y^4+x^4+3y^4+800-100xy^2-100x^2y$. Let $\a_1=(0,0),\a_2=(4,0),\a_3=(0,4),\a_4=(4,4)$ and $\b_1=(2,1),\b_2=(1,2)$. 
There are two simplices which cover $\b_1$: one with vertices $\{\a_1,\a_2\,\a_3\}$, denoted by $\Delta_1$, and one with vertices $\{\a_1,\a_2\,\a_4\}$, denoted by $\Delta_2$. 
There are two simplices which cover $\b_2$: $\Delta_1$ and one with vertices $\{\a_1,\a_3\,\a_4\}$, denoted by $\Delta_3$. One can check that $f$ admits a SONC decomposition $f=g_1+g_2$, where $g_1=20x^4y^4+x^4+400-100x^2y$, supported on $\Delta_2$, and $g_2=30x^4y^4+3y^4+400-100xy^2$, supported on $\Delta_3$, are both nonnegative circuit polynomials. Hence the simplex $\Delta_1$ is not needed in this SONC decomposition of $f$.
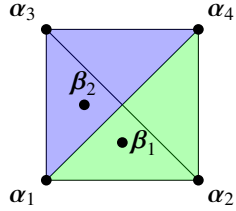
\begin{center}
\begin{tikzpicture}
\draw (0,0)--(0,2);
\draw (0,0)--(2,0);
\draw (2,0)--(2,2);
\draw (0,2)--(2,2);
\draw (0,0)--(2,2);
\draw (0,2)--(2,0);
\fill[fill=green,fill opacity=0.3] (0,0)--(2,0)--(2,2)--(0,0);
\fill[fill=blue,fill opacity=0.3] (0,0)--(0,2)--(2,2)--(0,0);
\fill (0,0) circle (2pt);
\node[below left] (1) at (0,0) {$\a_1$};
\fill (2,0) circle (2pt);
\node[below right] (2) at (2,0) {$\a_2$};
\fill (0,2) circle (2pt);
\node[above left] (3) at (0,2) {$\a_3$};
\fill (2,2) circle (2pt);
\node[above right] (4) at (2,2) {$\a_4$};
\fill (1,0.5) circle (2pt);
\node[right] (5) at (1,0.5) {$\b_1$};
\fill (0.5,1) circle (2pt);
\node[above] (6) at (0.5,1) {$\b_2$};
\end{tikzpicture}
\end{center}
\end{example}
A recent study by \cite{papp} proposes a systematic method to compute an optimal simplex cover. It would be worth trying to combine this framework with our SOC characterization for SONC cones to achieve a more accurate algorithm. Here we rely on a heuristics to compute a set of simplices with vertices coming from $\Lambda(f)$ and that covers $\Gamma(f)$.  
For $\b\in\Gamma(f)$ and $\a_0\in\Lambda(f)$, define an auxiliary linear program:
\begin{align*}
\textrm{SimSel}(\b,\Lambda(f),\a_0):=\,&\arg\max\quad\lambda_{\a_0}\\
&\,\textrm{s.t. }\left\{\sum_{\a\in\Lambda(f)}\lambda_{\a}\cdot\a=\b,\sum_{\a\in\Lambda(f)}\lambda_{\a}=1,\lambda_{\a}\ge0, \forall\a\in\Lambda(f)\right\}.
\end{align*}

Following \cite{se}, we can ensure that the output of $\textrm{SimSel}(\b,\Lambda(f),\a_0)$\footnote{If $\b$ lies on a face of $\New(f)$, then we assume $\a_0$ lies on the same face.} corresponds to a trellis which contains $\a_0$ and covers $\b$. 
%

\begin{algorithm}
\renewcommand{\algorithmicrequire}{\textbf{Input:}}
\renewcommand{\algorithmicensure}{\textbf{Output:}}
\caption{{\tt SimplexCover($\Lambda(f), \Gamma(f)$)}}\label{alg:simplexcover}
\begin{algorithmic}[1]
\REQUIRE
$\Lambda(f), \Gamma(f)$
\ENSURE
$\{(\TT_k,\b_k)\}_k$: a set of circuits such that $\cup_{k}\TT_k=\Lambda(f)$ and $\{\b_k\}_k=\Gamma(f)$
\STATE $U:=\Lambda(f), V:=\Gamma(f)$, $k:=0$;
\WHILE{$U\ne\emptyset$ and $V\ne\emptyset$}
\STATE $k:=k+1$;
\STATE Choose $\a_0\in U$ and $\b_k\in V$;
\STATE $\boldsymbol{\lambda}:=\textrm{SimSel}(\b_k,\Lambda(f),\a_0)$;
\STATE $\TT_k:=\{\a\in\Lambda(f)\mid\lambda_{\a}>0\}$;
\STATE $U:=U\backslash\TT_k$, $V:=V\backslash\{\b_k\}$;
\ENDWHILE
\IF{$V\ne\emptyset$}
\WHILE{$V\ne\emptyset$}
\IF{$U=\emptyset$}
\STATE $U:=\Lambda(f)$;
\ENDIF
\STATE $k:=k+1$;
\STATE Choose $\TT_0\in U$ and $\b_k\in V$;
\STATE $\boldsymbol{\lambda}:=\textrm{SimSel}(\b_k,\Lambda(f),\a_0)$;
\STATE $\TT_k:=\{\a\in\Lambda(f)\mid\lambda_{\a}>0\}$;
\STATE $U:=U\backslash\TT_k$, $V:=V\backslash\{\b_k\}$;
\ENDWHILE
\ELSE
\WHILE{$U\ne\emptyset$}
\IF{$V=\emptyset$}
\STATE $V:=\Gamma(f)$;
\ENDIF
\STATE $k:=k+1$;
\STATE Choose $\a_0\in U$ and $\b_k\in V$;
\STATE $\boldsymbol{\lambda}:=\textrm{SimSel}(\b_k,\Lambda(f),\a_0)$;
\STATE $\TT_k:=\{\a\in\Lambda(f)\mid\lambda_{\a}>0\}$;
\STATE $U:=U\backslash\TT_k$, $V:=V\backslash\{\b_k\}$;
\ENDWHILE
\ENDIF
\STATE \textbf{return} $\{(\TT_k,\b_k)\}_k$
\end{algorithmic}
\end{algorithm}

Suppose $\tilde{f}=\sum_{\a\in\Lambda(f)}c_{\a}\x^{\a}-\sum_{\b\in\Gamma(f)}d_{\b}\x^{\b}\in\R[\x]$ and assume that $\tilde{\A}=\{\a\in\Lambda(f)\mid\a\notin\cup_{\b\in\Gamma(f)}\cup_{\TT\in\CC(\b)}\TT\}=\varnothing$. The procedure {\tt SimplexCover} stated in Algorithm \ref{alg:simplexcover} allows one to compute a simplex cover $\{(\TT_k,\b_k)\}_{k=1}^l$ for $\tilde{f}$. 
For each $k$, let $M_{k}$ be a $\TT_{k}$-rational mediated set containing $\b_k$ and $s_k=\#M_k\backslash\TT_{k}$. 
For each $\bu_i^k\in M_k\backslash\TT_{k}$, let us write  $\bu_i^k=\frac{1}{2}(\bv_i^k+\bw_i^k)$. Let $\mathbf{K}$ be the $3$-dimensional rotated SOC, i.e.,
\begin{equation}
\mathbf{K}:=\{(a,b,c)\in\R^3\mid2ab\ge c^2,a\ge0,b\ge0\}.
\end{equation}

Then we can relax (SONC-PN) to a SOCP problem as follows:{\small
\begin{equation}\label{eq:soncsocp}
(\textrm{SONC-SOCP}):\quad\begin{cases}
\sup&\xi\\
\textrm{s.t.}&\tilde{f}(\x)-\xi=\sum_{k=1}^l\sum_{i=1}^{s_k}(2a_i^k\x^{\bv_i^k}+b_i^k\x^{\bw_i^k}-2c_i^k\x^{\bu_i^k}),\\
&(a_i^k,b_i^k,c_i^k)\in\mathbf{K},\quad\forall i,k.
\end{cases}
\end{equation}}
Let us denote by $\xi_{socp}$ the optimal value of (\ref{eq:soncsocp}). Then, we have $\xi_{socp}\le\xi_{sonc}\le\xi^*$.

\begin{remark}
\label{quality}
The quality of obtained SONC lower bounds depends on two successive steps: the relaxation to the corresponding PN-polynomial (from $\xi^*$ to $\xi_{sonc}$) and the relaxation to a specific simplex cover (from $\xi_{sonc}$ to $\xi_{socp}$). 
The loss of bound-quality at the second step can be improved by computing a more optimal simplex cover. Nevertheless, it may happen that the loss of bound-quality at the first step is already big, as shown in Example \ref{ex:quality}, which indicates that the gap between nonnegative polynomials and SONC PN-polynomials (see figure \ref{fig-poly}) may greatly affect the quality of SONC lower bounds.
\end{remark}

\begin{example}
\label{ex:quality}
Let $f=1+x_1^4+x_2^4-x_1x_2^2-x_1^2x_2+5x_1x_2$. Since $\Lambda(f)$ forms a trellis, the simplex cover for $f$ is unique. One obtains  $\xi_{socp}=\xi_{sonc}\approx-6.916501$ while $\xi^*\approx-2.203372$. 
Hence the relative optimality gap is near $214\%$.
\end{example}

\begin{figure}[htbp]
\begin{center}
\begin{tikzpicture} 
\node at (-3.2,0) {{\footnotesize PSD polynomials}};
\node at (3.2,0) {{\footnotesize PN-polynomials}};
\node at (-0.6,0) {{\footnotesize SONC polynomials}};
\draw[color=blue] (-1.7,0) ellipse (3 and 2);
\draw[color=red] (1.7,0) ellipse (3 and 2);
\draw[color=green] (-0.6,0) circle (1.4);
\end{tikzpicture}
\end{center}
\caption{Relationship of different classes of polynomials}\label{fig-poly}
\end{figure}
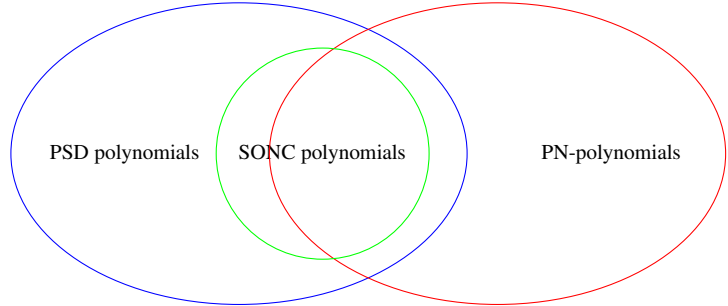

\section{Exact SOBS decompositions of SONC polynomials}
\label{sec:exact}

\subsection{The procedure ${\tt ExactSOBS}$}

Here, we present a procedure, called ${\tt ExactSOBS}$, to compute exact rational SOBS decompositions (with rational exponents) of PN-polynomials lying in the interior of the SONC cone (if $f$ itself is not a PN-polynomial, then by Lemma \ref{sec4-lm1}, an exact rational SOBS decompositions of its associated PN-polynomial $\tilde{f}$ also provides a nonnegativity certificate for $f$).
The procedure, stated in Algorithm \ref{alg:roundproject}, is a rounding-projection procedure, in the spirit of the one by \cite{pe} to obtain exact SOS decompositions for polynomials in the interior of the SOS cone and by \cite{ma19} to obtain exact SAGE decompositions of polynomials in the interior of the SAGE cone. 

\begin{algorithm}
\renewcommand{\algorithmicrequire}{\textbf{Input:}}
\renewcommand{\algorithmicensure}{\textbf{Output:}}
\newcommand{\roundfun}[2]{\texttt{round}(#1,#2)}
\caption{${\tt ExactSOBS}(f, \hat\delta, \tilde\delta)$}\label{alg:roundproject}
\begin{algorithmic}[1]
\REQUIRE
A SONC PN-polynomial $f=\sum_{\a\in\A}f_{\a}\x^{\a}\in\Q[\x]$ with $\tilde{\A}=\varnothing$, a precision parameter $\tilde \delta$ for the SOCP solver, a rounding precision $\hat \delta$
\ENSURE
An exact SOBS decomposition of $f$
\STATE Determine the sets $\Lambda(f),\Gamma(f)$ such that $f=\sum_{\a\in\Lambda(f)}c_{\a}\x^{\a}-\sum_{\b\in\Gamma(f)}d_{\b}\x^{\b}$;
\STATE $\{(\TT_k,\b_k)\}_{k=1}^l:={\tt SimplexCover}(\Lambda(f),\Gamma(f))$; \label{line:simplexcover}
\FOR {$k=1:l$}
\STATE $\{(\bu_i^k,\bv_i^k,\bw_i^k)\}_{i=1}^{s_k}:={\tt MedSet}(\TT_k,\b_k)$; \label{line:medset}
\ENDFOR
\STATE Obtain $\{(\tilde a_i^k,\tilde b_i^k,\tilde c_i^k)\}_{i,k}$ by solving (\textrm{SONC-SOCP}) with $\xi=0$ at precision $\tilde{\delta}$; \label{line:socp}
\STATE $\{(\hat{a}_i^k,\hat{b}_i^k,\hat{c}_i^k)\}_{i,k}:= \roundfun{\{(\tilde a_i^k, \tilde b_i^k, \tilde c_i^k)\}_{i,k}}{\tilde \delta}$; \COMMENT{rounding step} \label{line:round}
\STATE $B:=\cup_{k=1}^l\cup_{i=1}^{s_k}\{\bu_i^k,\bv_i^k,\bw_i^k\}$;
\FOR{$\gg\in B$}
\STATE $\eta(\gg):=\#\{(i,k)\mid\gg \in \{ \bu_i^k,\bv_i^k,\bw_i^k \} ,i=1,\ldots,s_k,k=1,\ldots,l\}$;
\ENDFOR
\FOR {$k=1:l,i=1:s_k$}
\STATE ${a}_i^k:=\hat{a}_i^k-\frac{1}{2\eta(\bv_i^k)}\Bigl( \sum_{\bu_p^q=\bv_i^k}2\hat{a}_p^q+\sum_{\bv_p^q=\bv_i^k}\hat{b}_p^q-\sum_{\bw_p^q=\bv_i^k}2\hat{c}_p^q-f_{\bv_i^k}\Bigr)$; \COMMENT{projection step} \label{line:proj}
\STATE ${b}_i^k:=\hat{b}_i^k-\frac{1}{\eta(\bw_i^k)}\Bigl( \sum_{\bu_p^q=\bw_i^k}2\hat{a}_p^q+\sum_{\bv_p^q=\bw_i^k}\hat{b}_p^q-\sum_{\bw_p^q=\bw_i^k}2\hat{c}_p^q-f_{\bw_i^k}\Bigr)$; \COMMENT{projection step}
\STATE ${c}_i^k:=\hat{c}_i^k+\frac{1}{2\eta(\bu_i^k)}\Bigl( \sum_{\bu_p^q=\bu_i^k}2\hat{a}_p^q+\sum_{\bv_p^q=\bu_i^k}\hat{b}_p^q-\sum_{\bw_p^q=\bu_i^k}2\hat{c}_p^q-f_{\bu_i^k}\Bigr)$; \COMMENT{projection step} \label{line:projend}
\ENDFOR
\FOR {$k=1:l$}
\STATE $f_k:=\sum_{i=1}^{s_k} (2{a}_i^k\x^{\bv_i^k}+{b}_i^k\x^{\bw_i^k}-2{c}_i^k\x^{\bu_i^k})$;
\ENDFOR
\STATE\textbf{return} $\sum_{k=1}^lf_k$;
\end{algorithmic}
\end{algorithm}

After computing a simplex cover in Line \ref{line:simplexcover} with Algorithm \ref{alg:simplexcover} and the corresponding rational mediated sets in Line \ref{line:medset} with Algorithm \ref{alg:medset}, the procedure calls a SOCP solver in Line \ref{line:socp} to compute a $\tilde{\delta}$-approximation $(\tilde{a}_i^k, \tilde{b}_i^k,  \tilde{c}_i^k)_{i,k}$ of the coefficients involved in the SOBS decomposition of $f$ from \eqref{eq:soncsocp}.
This approximation is then rounded to rational points in Line \ref{line:round} with a prescribed accuracy of $\hat{\delta}$.
The three projection steps from Line \ref{line:proj} to Line \ref{line:projend} ensure that the rational points $({a}_i^k, {b}_i^k, {c}_i^k)_{i,k}$ satisfy exactly the equality constraints from \eqref{eq:soncsocp}, namely that for each $k=1,\dots,\ell$ and each $i=1,\dots,s_k$, one has 
\begin{align}
\label{eq:eqsocp}
f_{\gg} = \left( \sum_{ \underset{ \bv_p^q=\gg}{p,q}}2a_p^q+\sum_{\underset{ \bw_p^q=\gg}{p,q}}b_p^q-\sum_{\underset{ \bu_p^q=\gg}{p,q}}2c_p^q\right) \,,
\end{align}
for each $\gg \in \{\bu_i^k,\bv_i^k,\bw_i^k\}_{i,k}$.
As proved later on, if both numerical SOCP solver accuracy $\tilde \delta$ and rounding accuracy $\hat{\delta}$ are high enough, then each triple $({a}_i^k, {b}_i^k, {c}_i^k)$ belongs to the second-order cone $\mathbf{K} = \{(a,b,c): a,b \geq 0 \,, 2 a b \geq c^2 \}$.
\subsection{Arithmetic complexity}

The {\textit{bit size}} of $i \in \Z$ is denoted by ${\tau(i)} := \lfloor \log_2 (|i|) \rfloor + 1$ with ${\tau(0)} := 1$.  
Given $i \in \Z$ and $j \in \Z \backslash \{0\}$ with gcd$(i,j) = 1$, we define ${\tau(i/j)} := \max \{\tau(i), \tau(j)\}$. 
For two mappings $g,h : \N^m \to \R$, we use the notation {$g(v) \in \bigo{(h(v))}$} to state the existence of $i \in \N$ such that $g(v) \leq i h(v)$, for all $v \in \N^m$.
The expression ``$g(v) \in \bigotilde{(h(v))}$'' means that there exists $c \in \N$ such that $g(v) \leq  h(v) \log_2 (h(v))^c$, for all $v \in \N^m$.

\begin{lemma}\label{bound-denom}
Let $f\in\R[\x]$ be a polynomial of degree $d$ and let $D$ be the maximal value of the denominators involved in the triples $\{(\bu_i^k,\bv_i^k,\bw_i^k)\}_{i,k}$ in \eqref{eq:soncsocp}. Then $D\le (1+nd^2)^{n+1}$.
\end{lemma}
\begin{proof}
Consider any circuit $(\TT,\b)$ such that $\TT\subseteq\Lambda(f)$ and $\b\in\Gamma(f)$. Without loss of generality, assume that $\TT=\{\a_1,\ldots,\a_{n+1}\}$. We write $\b=\sum_{i=1}^{n+1}\frac{q_i}{p}\a_i$, where $p=\sum_{i=1}^{n+1}q_i$, $p,q_i\in\N\backslash\{0\}$, gcd$(p,q_1,\ldots,q_{n+1})=1$.

First, one can easily see that the denominators involved in the output $\{(\bu_i,\bv_i,\bw_i)\}_i$ of the algorithm ${\tt LMedSet}(\a',\a'',\b')$ do not exceed the common denominator of $\a',\a'',\b'$. Thus it suffices to consider the denominators of $\b_i$'s appearing in the algorithm ${\tt MedSet}$, which are clearly bounded from above by $p$. Note that $\bm\lambda=(\frac{q_1}{p},\ldots,\frac{q_{n+1}}{p})$ is the unique solution to the system of linear equations:
\begin{equation*}
    \begin{bmatrix}
    1&\cdots&1\\ \a_1&\cdots&\a_{n+1}
    \end{bmatrix}\bm\lambda=\begin{bmatrix}
    1\\ \b
    \end{bmatrix}.
\end{equation*}
Hence $p\le\det(\begin{bmatrix}
    1&\cdots&1\\ \a_1&\cdots&\a_{n+1}
    \end{bmatrix})\le(1+nd^2)^{\frac{n+1}{2}}$. 
The latter inequality follows from Hadamard's inequality. As a result, $D\le p^2\le (1+nd^2)^{n+1}$.
\end{proof}

For later purpose, we recall the result which bounds the roots of univariate polynomials with integer coefficients:
\begin{lemma}{~\cite[Theorem~4.2 (ii)]{Mignotte1992}}
	\label{lemma:mignotte}
	Let $f \in \Z[x]$ of degree $d$, with coefficient bit size bounded from above by $\tau$. 
	If $f(\tilde x) = 0$ and $\tilde x \neq 0$, then $\frac{1}{2^\tau + 1} \leq |\tilde x| \leq 2^\tau + 1$.
\end{lemma}
\begin{lemma}
\label{lemma:sonc_distance}
Let $f=\sum_{\a\in\A}f_{\a}\x^{\a}\in\Q[\x]$ be a SONC PN-polynomial of degree $d \geq 2$ with $\tau = \tau (f)$. 
Assume that $\tilde{\A}=\varnothing$ and $f$ lies in the interior of the SONC cone.
Let $D$ be the maximal value of the denominators involved in the triples $\{(\bu_i^k,\bv_i^k,\bw_i^k)\}_{i,k}$ in \eqref{eq:soncsocp}, bounded as in Lemma \ref{bound-denom}.
Then, there exists $N \in \N$ such that for $\varepsilon := 2^{-N}$, $f - \varepsilon \sum_{k=1}^l \sum_{i=1}^{s_k} (\x^{\bu_i^k} + \x^{\bv_i^k} + \x^{\bw_i^k})$ has an SOBS decomposition (with rational exponents), with $N \leq \tau (\varepsilon) \in \bigo{(\tau \cdot (4d D+6)^{3n+3})}$, thus $N \leq \tau (\varepsilon) \in \bigo{(\tau \cdot (5d )^{3n+3} (1+nd^2)^{3 (n+1)^2} )}$.
\end{lemma}
\begin{proof}
Let us define $g(\x) := f(\x^D)$.
Note that the degree of $g$ is at most $D d$. 
In addition, for each each $\gg \in \{\bu_i^k,\bv_i^k,\bw_i^k\}_{i,k}$ involved in the SOBS decomposition of $f$, one has $\x^{D \gg} \in \R[\x]$. 
Thus $g$ is a sum of binomial squares and one has $||D\gg||_1\leq D d$.
Since the polynomial $f$ is in the interior of the SONC cone, then the polynomial $g$ is also in the interior of the SONC cone. 
Then by definition, there exists $N \in \N$  such that for $\varepsilon := 2^{-N}$, one has $g - \varepsilon \sum_{k=1}^l \sum_{i=1}^{s_k} (\x^{D \bu_i^k} + \x^{D \bv_i^k} + \x^{D \bw_i^k}) \in \SONC$.
Let us define the polynomial $h(\x,z) := f(\x^D) - z \sum_{k=1}^l \sum_{i=1}^{s_k} (\x^{D \bu_i^k} + \x^{D \bv_i^k} + \x^{D \bw_i^k})$ , 
and consider the algebraic set $V$ defined by:
\[
V := \left\{ (\x,z) \in \R^{n+1} : h(\x,z) = \frac{\partial h }{\partial x_1} = \dots = \frac{\partial h}{\partial x_n} = 0 \right\} \,.
\] 
The degree of $h$ is also $D d$. 
Using Proposition~A.1 from  \cite{magron2018exact}, there exists a polynomial in $\Z[z]$ of degree less than $(D d +1)^{n+1}$ with coefficients of bit size less than $\tau \cdot (4 D d + 6)^{3 n + 3}$ such that its set of real roots contains the projection of $V$ on the $z$-axis. 
By Lemma~\ref{lemma:mignotte}, it is enough to take $N \leq \tau \cdot (4 D d + 6)^{3 n + 3}$.
By Lemma~\ref{bound-denom}, one has $D \leq (1+nd^2)^{n+1}$. The desired complexity bound follows from the fact that  $d\geq 2$ implies that $6 \leq d(1+nd^2)^{n+1}$.
By Theorem \ref{sec4-thm}, the polynomial $h(\x,\varepsilon)$ admits an SOBS decomposition. 
Therefore $h(\x^{1/D},\varepsilon)$ also admits an SOBS decomposition, the desired result.
\end{proof}
\begin{lemma}\label{bound-triple}
Let $f\in\R[\x]$ be as in Lemma \ref{lemma:sonc_distance} with $t$ terms. 
Let $L$ be the total number of triples $(\bu_i^k,\bv_i^k,\bw_i^k)$ appearing in \eqref{eq:soncsocp}, i.e., $L=\sum_{k=1}^l s_k$.
Then $L<\frac{1}{8}tn((n+1)\log_2(1+nd^2)+3)^2 \in \bigotilde{(d^n n^3)}$.
\end{lemma}
\begin{proof}
First note that the number of circuits produced by the procedure ${\tt SimplexCover}(\Lambda(f), \Gamma(f))$ is less than $t$. Consider any circuit $(\TT,\b)$ such that $\TT\subseteq\Lambda(f)$ and $\b\in\Gamma(f)$. Without loss of generality, assume that $\TT=\{\a_1,\ldots,\a_{n+1}\}$. We write $\b=\sum_{i=1}^{n+1}\frac{q_i}{p}\a_i$, where $p=\sum_{i=1}^{n+1}q_i$, $p,q_i\in\N\backslash\{0\}$, gcd$(p,q_1,\ldots,q_{n+1})=1$. Then by the proof of Lemma \ref{bound-denom}, $p\le(1+nd^2)^{\frac{n+1}{2}}$. For each triple $(\a_k,\b_k,\b_{k-1})$, $1\le k\le n-1$ or $(\a_{n},\a_{n+1},\b_{n-1})$ appearing in the algorithm ${\tt MedSet}$, by Lemma \ref{sec4-lm3} and the proof of Lemma \ref{sec4-lm2}, the number of triples $(\bu_i,\bv_i,\bw_i)$ produced by the algorithm ${\tt LMedSet}$ is less than $\frac{1}{2}(\log_2(p)+\frac{3}{2})^2$. Putting all above together, we have $L<tn\cdot\frac{1}{2}(\log_2((1+nd^2)^{\frac{n+1}{2}})+\frac{3}{2})^2=\frac{1}{8}tn((n+1)\log_2(1+nd^2)+3)^2$.
Note that the number of terms $t$ is less than $2d^n$:
\[
t = \binom{n+d}{n} = \frac{(n+d)\dots (n+1)}{n!} = \left(1+\frac{d}{n}\right) \left(1+\frac{d}{n-1}\right) \dots (1+d) \leq d^{n-1}(1+d) \leq 2 d^n\,,
\]
yielding the final complexity result.
\end{proof}

\begin{remark}
By Lemma \ref{bound-triple}, the number of second-order cones used in (\textrm{SONC-SOCP}) \eqref{eq:soncsocp} scales linearly in the number of terms of the polynomial.
\end{remark}

\begin{theorem}
\label{th:arithcost}
Let $f, D$ be as in Lemma \ref{lemma:sonc_distance}. Let $L$ be the total number of triples $(\bu_i^k,\bv_i^k,\bw_i^k)$ appearing in \eqref{eq:soncsocp}, i.e., $L=\sum_{k=1}^l s_k$.
There exist $\hat{\delta}$, $\tilde{\delta}$ of bit size less than $\bigo{(\tau \cdot (4 D d + 6)^{3n+3})}$, such that ${\tt ExactSOBS}(f, \hat{\delta},\tilde{\delta})$ terminates and outputs a rational SOBS decomposition of $f$ within $\bigo{( (\tau (L)+ \tau \cdot (4  Dd + 6)^{3n+3}) \, L^{3.5} )}$ arithmetic operations, yielding a total number of 
 \[
 \bigotilde{\left(  \tau \cdot 5^{3n+3} d^{6.5n+3} n^{10.5}(1+nd^2)^{3 (n+1)^2}  \right)}
 \]
 arithmetic operations.
\end{theorem}
\begin{proof}
Let $\varepsilon$ be as in Lemma \ref{lemma:sonc_distance}. 
Then, there exist sequences of triples $\{(\overline{a}_i^k,  \overline{b}_i^k, \overline{c}_i^k)\}_{i,k}$ as in \eqref{eq:soncsocp} such that 
\begin{align*}
f - \varepsilon \sum_{k=1}^l \sum_{i=1}^{s_k} (\x^{\bu_i^k} + \x^{\bv_i^k} + \x^{\bw_i^k}) = & \sum_{k=1}^l\sum_{i=1}^{s_k}(2  \overline  a_i^k\x^{\bv_i^k}+ \overline{b}_i^k\x^{\bw_i^k}-2  \overline c_i^k\x^{\bu_i^k}), \quad(\overline a_i^k, \overline b_i^k, \overline   c_i^k)\in\mathbf{K},\forall i,k.
\end{align*}
(1) Let us define $ \delta := \min \{\frac{3 \varepsilon^2}{4},\frac{\varepsilon}{2} \}$, $\tilde a_i^k := \overline a_i^k + \frac{\varepsilon}{2}$, $\tilde b_i^k := \overline b_i^k + \varepsilon$ and $\tilde c_i^k := \overline c_i^k - \frac{\varepsilon}{2}$.
Let us show that $\{(\tilde{a}_i^k, \tilde{b}_i^k, \tilde{c}_i^k)\}_{i,k}$ is a strictly feasible solution of \eqref{eq:soncsocp} (with $\xi = 0$) such that $\tilde a_i^k \geq  \delta$, $b_i^k \geq \delta$ and $2 \tilde a_i^k \tilde b_i^k - (\tilde c_i^k)^2 \geq  \delta$.\\
According to these definitions, one has
\[
f  =  \sum_{k=1}^l\sum_{i=1}^{s_k}(2  \tilde a_i^k\x^{\bv_i^k}+ \tilde b_i^k\x^{\bw_i^k}-2  \tilde c_i^k\x^{\bu_i^k})\,,
\]
and 
\begin{align}
\label{eq:eqsocptilde}
 \sum_{ \underset{ \bv_p^q=\gg}{p,q}}2\tilde{a}_p^q+\sum_{\underset{ \bw_p^q=\gg}{p,q}}\tilde{b}_p^q-\sum_{\underset{ \bu_p^q=\gg}{p,q}}2\tilde{c}_p^q - f_{\gg}  = 0 \,,
\end{align}
for each $\gg \in \{\bu_i^k,\bv_i^k,\bw_i^k\}_{i,k}$.
Since $\overline a_i^k, \overline b_i^k \geq 0$, one has $\tilde a \geq \frac{\varepsilon}{2} \geq  \delta$, $\tilde b \geq \varepsilon \geq  \delta$.
In addition, for all $(a,b,c) \in \mathbf{K}$, one has $2 a b - c^2 \geq 0$, which implies that $- \sqrt{2a b} \leq c \leq \sqrt{2 ab}$ and $2 a + b + c \geq 2 a + b - \sqrt{2 a b} = (\sqrt{2 a} - \sqrt{b})^2 + \sqrt{2a b} \geq 0$.
Hence, one has
\[
2 \tilde a_i^k \tilde b_i^k - (\tilde c_i^k)^2 = (2 \overline a_i^k + \varepsilon)(\overline b_i^k + \varepsilon) - (\overline c_i^k - \frac{\varepsilon}{2})^2 = 2 \overline a_i^k  \overline b_i^k - (\overline c_i^k)^2 + (2 \overline a_i^k + \overline b_i^k + \overline c_i^k) \varepsilon + \frac{3 \varepsilon^2}{4} \geq \frac{3 \varepsilon^2}{4} \geq  \delta \,.
\]
With this choice, one has $\tau( \delta) \leq 2 \tau (\varepsilon) + \tau(\frac{3}{4}) $. So 
$\delta$ has bit size less than $\bigo{(\tau \cdot (4 D d + 6)^{3n+3})}$.\\
%
(2) \if{Let us fix $\tilde \delta > 0$. For any $\tilde \delta$-feasible solution $(\tilde{\A}_i^k, \tilde{b}_i^k, \tilde{c}_i^k)_{i,k}$ for \eqref{eq:soncsocp} (with $\xi = 0$), instead of having \eqref{eq:eqsocp}, one has
\begin{align}
\label{eq:eqsocptilde}
\left|  \sum_{ \underset{ \bv_p^q=\gg}{p,q}}2\tilde{\A}_p^q+\sum_{\underset{ \bw_p^q=\gg}{p,q}}\tilde{b}_p^q-\sum_{\underset{ \bu_p^q=\gg}{p,q}}2\tilde{c}_p^q - f_{\gg} \right| \leq \tilde \delta \,,
\end{align}
for each $\gg \in \{\bu_i^k,\bv_i^k,\bw_i^k\}$.
}\fi
Let us assume that one relies in Algorithm \ref{alg:roundproject} on a rounding procedure with precision $\hat \delta$ so that $|\tilde a_i^k - \hat a_i^k| \leq \hat \delta$ (and similarly for $\tilde b_i^k$ and $\tilde c_i^k$).
By using the triangular inequality, one obtains
\begin{align}
\label{eq:eqsocphat}
\left|  \sum_{ \underset{ \bv_p^q=\gg}{p,q}}2\hat{a}_p^q+\sum_{\underset{ \bw_p^q=\gg}{p,q}}\hat{b}_p^q-\sum_{\underset{ \bu_p^q=\gg}{p,q}}2\hat{c}_p^q - f_{\gg} \right| \leq 
  \sum_{ \underset{ \bv_p^q=\gg}{p,q}}2|\hat{a}_p^q - \tilde{\A}_p^q | +
    \sum_{ \underset{ \bw_p^q=\gg}{p,q}}|\hat{b}_p^q - \tilde{b}_p^q | +
      \sum_{ \underset{ \bu_p^q=\gg}{p,q}}2|\hat{c}_p^q - \tilde{c}_p^q | 
  \leq 5 \eta(\gg) \hat \delta  \,,
\end{align}
for each $\gg \in \{\bu_i^k,\bv_i^k,\bw_i^k\}_{i,k}$.
Let us consider the projection steps of Algorithm \ref{alg:roundproject}. 
For $a_i^k$, $b_i^k$ and $c_i^k$ defined from Line \ref{line:proj} to Line \ref{line:projend}, one has 
\begin{align}
\label{eq:hatdelta}
a_i^k \geq \hat a_i^k - \frac{5}{2} \hat \delta    \,, \\
b_i^k \geq \hat b_i^k - 5 \hat \delta  \,, \\
 \tilde c_i^k + \frac{5}{2}\hat \delta  \geq c_i^k  \,.
\end{align}
Let us choose  $ \hat \delta =  \frac{\varepsilon}{5}$.
Therefore in the worse case scenario, one has $a_i^k = \tilde a_i^k - \frac{\varepsilon}{2} = \overline a_i^k$, $b_i^k = \tilde b_i^k - \varepsilon = \overline b_i^k$ and $c_i^k = \tilde c_i^k + \frac{\varepsilon}{2} = \overline c_i^k$.

To conclude, if one solves the SOCP problem in Algorithm \ref{alg:roundproject} at precision $\tilde\delta$ at least $\delta$ and rounding accuracy $\hat\delta$, then the output will always provide a valid SOBS decomposition.
Both accuracy parameters have bit size less than $\bigo{(\tau \cdot (4 D d + 6)^{3n+3})}$.

The arithmetic complexity of the procedure to solve the SOCP problem at accuracy $\tilde{\delta}$ is derived from \S~4.6.2 in \cite{ben2001lectures}.
There are $3 L $ inequality constraints and $3 L $ equality constraints (which can themselves be cast as $6L$ inequality constraints), involving $3 L$ variables, yielding an overall complexity of $\bigo{\left(\sqrt{9L+1} \cdot 3 L \cdot  ((3L)^2 + 9L + 9L) \cdot \log \frac{L + \delta^2}{\delta} \right)}$, which leads to a number of $\bigo{\left( (\tau (L)+ \tau \cdot (4  Dd + 6)^{3n+3}) \, L^{3.5} \right)}$ arithmetic operations.
Using the bound on $L$ stated in Lemma \ref{bound-triple}, we obtain a total number of $\bigotilde{\left(  \tau \cdot 5^{3n+3} d^{6.5n+3} n^{10.5}(1+n d^2)^{3 (n+1)^2}  \right)}$ arithmetic operations.
The simplex cover algorithm ${\tt SimplexCover}$ solves $\bigo{(t)}$ linear programs involving $\bigo{(t)}$  variables. 
The complexity of a linear program is $O(t^{2.5})$. So the total complexity is $O(t^{3.5})$.
Let $p$ be as in Lemma \ref{bound-denom} which is bounded by $(1+nd^2)^{\frac{n+1}{2}}$. 
The procedure ${\tt MedSeq}(p,q)$ requires $\bigo{(\log_2 p )}$ iterations and each iteration needs $\bigo{(\log_2 p)}$ arithmetic operations. 
Since we have $\bigo{(t)}$  circuits and each circuit corresponds to (at most) $n$ mediated sequences, the total complexity is $\bigo{(t n \log_2(p)^2)}=\bigo{(t n^3 \log_2(1+n d^2)^2)}$.
Overall, this shows that the total number of arithmetic operations required by these two steps of ${\tt ExactSOBS}$ have a negligible cost by comparison with the numerical SOCP procedure, yielding the desired result.
\end{proof}

\section{Numerical experiments}
\label{sec:benchs}
In this section we present numerical results of the proposed algorithms for unconstrained POPs. 
Our tool, called {\tt SONCSOCP}, implements the simplex cover procedure (see Algorithm \ref{alg:simplexcover}) as well as the procedure ${\tt MedSet}$ (see Algorithm \ref{alg:medset}) computing rational mediated sets and computes the optimal value $\xi_{socp}$ of the SOCP \eqref{eq:soncsocp} with Mosek; see \cite{mosek}. 
All experiments were performed on an Intel Core i5-8265U@1.60GHz CPU with 8GB RAM memory and WINDOWS 10 system. {\tt SONCSOCP} is available at \href{https://github.com/wangjie212/SONCSOCP}{github:{\tt SONCSOCP}}.

Our benchmarks are issued from the database of randomly generated polynomials provided by Seidler and de Wolff in \cite{se}. 
Depending on the Newton polytope, these benchmarks are divided into three classes: the ones with standard simplices, the ones with general simplices and the ones with arbitrary Newton polytopes. See \cite{se} for the details on the construction of these polynomials.

\begin{table}[htbp]
\caption{The notation}\label{table1}
\begin{center}
\begin{tabular}{|c|c|}
\hline
$n$&number of variables\\
\hline
$d$&degree\\
\hline
$t$&number of terms\\
\hline
$l$&lower bound on the number of inner terms\\
\hline
opt&optimal value\\
\hline
time&running time in seconds\\
\hline
bit&bit size\\
\hline
\end{tabular}
\end{center}
\end{table}

\subsection{Computing a lower bound via SONC optimization}
In this subsection, we perform SONC optimization \eqref{eq:soncsocp} and compare the performance of {\tt SONCSOCP} with that of {\tt POEM}, which relies on the ECOS solver to solve geometric programs (see \cite{se} for more details). To measure the quality of a given lower bound $\xi_{lb}$, we rely on the ``{\tt local\_min}" function available in {\tt POEM} which computes an upper bound $\xi_{min}$ on the minimum of a  polynomial. The relative optimality gap is defined by 
$\frac{|\xi_{min}-\xi_{lb}|}{|\xi_{min}|}.$

\paragraph{\textbf{Standard simplex}}
For the standard simplex case, we take $10$ polynomials of different types (labeled by $N$). Running time and lower bounds obtained with {\tt SONCSOCP} and {\tt POEM} are displayed in Table \ref{tb2:standard}. Note that for polynomials with $\Lambda(\cdot)$ forming a trellis, the simplex cover is unique, thus the bounds obtained by {\tt SONCSOCP} and {\tt POEM} are the same theoretically, which is also reflected in Table \ref{tb2:standard}. 
For each polynomial, the relative optimality gap is less than $1\%$ and for $8$ out of $10$ polynomials, it is less than $0.1\%$ (see Figure \ref{fg2:standard}).

\begin{table}[htbp]
{\small
\begin{center}
\begin{tabular}{c|c|cccccccccc}
\multicolumn{2}{c|}{$N$}&$1$&$2$&$3$&$4$&$5$&$6$&$7$&$8$&$9$&$10$\\
\hline
\multicolumn{2}{c|}{$n$}&$10$&$10$&$10$&$20$&$20$&$20$&$30$&$30$&$40$&$40$\\
\multicolumn{2}{c|}{$d$}&$40$&$50$&$60$&$40$&$50$&$60$&$50$&$60$&$50$&$60$\\
\multicolumn{2}{c|}{$t$}&$20$&$20$&$20$&$30$&$30$&$30$&$50$&$50$&$100$&$100$\\
\hline
\multirow{2}*{time}&{\tt SONCSOCP}&$0.04$&$0.04$&$0.04$&$0.14$&$0.14$&$ 0.13$&$0.43$&$0.40$&$2.23$&$2.21$\\
\cline{3-12}
&{\tt POEM}&$0.26$&$0.27$&$0.26$&$0.43$&$0.44$&$0.42$&$1.78$&$1.79$&$2.20$&$2.25$\\
\hline
\multirow{2}*{opt}&{\tt SONCSOCP}&$3.52$&$3.52$&$3.52$&$2.64$&$2.64$&$2.64$&$2.94$&$2.94$&$4.41$&$4.41$\\
\cline{3-12}
&{\tt POEM}&$3.52$&$3.52$&$3.52$&$2.64$&$2.64$&$2.64$&$2.94$&$2.94$&$4.41$&$4.41$\\
\hline
\end{tabular}
\end{center}}
\caption{Results for the standard simplex case}\label{tb2:standard}
\end{table}

\begin{figure}[htbp]
\begin{center}
\begin{tikzpicture}
\begin{axis}[xlabel={$N$},
ylabel={running time (s)},
legend pos=north west]
\addplot[color=blue,mark=o]
coordinates{(1,0.04)(2,0.04)(3,0.04)(4,0.14)(5,0.14)(6,0.13)(7,0.43)(8,0.40)(9,2.23)(10,2.21)};
\addlegendentry{SONCSOCP}
\addplot[color=red,mark=star]
coordinates{(1,0.26)(2,0.27)(3,0.26)(4,0.43)(5,0.44)(6,0.42)(7,1.78)(8,1.79)(9,2.20)(10,2.25)};
\addlegendentry{POEM}
\end{axis}
\end{tikzpicture}
\end{center}
\caption{Running time for the standard simplex case}\label{fg1:standard}
\end{figure}
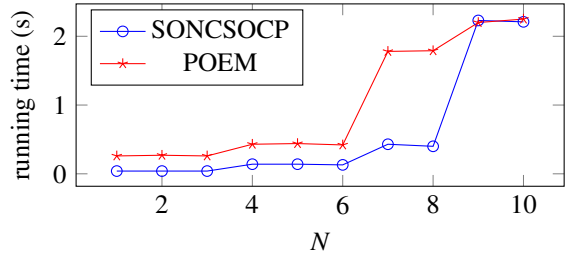

\begin{figure}[htbp]
\begin{center}
\begin{tikzpicture}
\begin{axis}[xlabel={$N$},
ylabel={Relative optimality gap (\%)},
legend pos=north east]
\addplot[color=blue,mark=o]
coordinates{(1,0.02)(2,0.13)(3,0.10)(4,0.0001)(5,0.0001)(6,0.004)(7,0.00001)(8,0.0001)(9,0.0001)(10,0.001)};
\addlegendentry{SONCSOCP}
\addplot[color=red,mark=star]
coordinates{(1,0.02)(2,0.13)(3,0.10)(4,0.0001)(5,0.0001)(6,0.004)(7,0.00001)(8,0.0001)(9,0.0001)(10,0.001)};
\addlegendentry{POEM}
\end{axis}
\end{tikzpicture}
\end{center}
\caption{Relative optimality gap for the standard simplex case}\label{fg2:standard}
\end{figure}
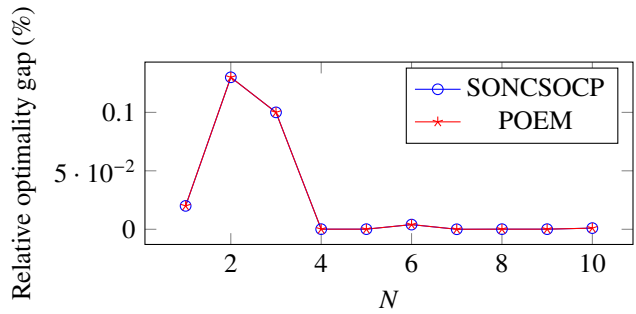

\paragraph{\textbf{General simplex}}
Here, we take $10$ polynomials of different types (labeled by $N$). 
Running time and lower bounds obtained with {\tt SONCSOCP} and {\tt POEM} are displayed in Table \ref{tb1:simplex}. 
As before, the SONC lower bounds obtained by {\tt SONCSOCP} and {\tt POEM} are the same. For each polynomial except for the one corresponding to $N = 7$, the relative optimality gap is within $30\%$, and for $6$ out of $10$ polynomials, the gap is below $1\%$ (see Figure \ref{fg2:simplex}). {\tt POEM} fails to obtain a lower bound for the instance $N = 10$ by returning $-$Inf.
Figure \ref{fg1:simplex} shows that, overall, the running times of {\tt SONCSOCP} and {\tt POEM} are close. 
{\tt SONCSOCP} is faster than {\tt POEM} for the instance $N = 6$, possibly because better performance is obtained when the degree is relatively low.

\begin{table}[htbp]
{\small
\begin{center}
\begin{tabular}{c|c|cccccccccc}
\multicolumn{2}{c|}{$N$}&$1$&$2$&$3$&$4$&$5$&$6$&$7$&$8$&$9$&$10$\\
\hline
\multicolumn{2}{c|}{$n$}&$10$&$10$&$10$&$10$&$10$&$10$&$10$&$10$&$10$&$10$\\
\multicolumn{2}{c|}{$d$}&$20$&$30$&$40$&$50$&$60$&$20$&$30$&$40$&$50$&$60$\\
\multicolumn{2}{c|}{$t$}&$20$&$20$&$20$&$20$&$20$&$30$&$30$&$30$&$30$&$30$\\
\hline
\multirow{2}*{time}&{\tt SONCSOCP}&$0.32$&$0.29$&$0.36$&$0.48$&$0.54$&$ 0.56$&$0.73$&$0.88$&$1.04$&$1.04$\\
\cline{2-12}
&{\tt POEM}&$0.28$&$0.31$&$0.31$&$0.31$&$0.43$&$0.74$&$0.75$&$0.74$&$0.72$&$0.76$\\
\hline
\multirow{2}*{opt}&{\tt SONCSOCP}&$1.18$&$0.22$&$0.38$&$0.90$&$0.06$&$4.00$&$-4.64$&$1.62$&$2.95$&$5.40$\\
\cline{2-12}
&{\tt POEM}&$1.18$&$0.22$&$0.38$&$0.90$&$0.06$&$4.00$&$-4.64$&$1.62$&$2.95$&$-$Inf\\
\end{tabular}
\end{center}}
\caption{Results for the general simplex case}\label{tb1:simplex}
\end{table}

\begin{figure}[htbp]
\begin{center}
\begin{tikzpicture}
\begin{axis}[xlabel={$N$},
ylabel={running time (s)},
legend pos=north west]
\addplot[color=blue,mark=o]
coordinates{(1,0.32)(2,0.29)(3,0.36)(4,0.48)(5,0.54)(6,0.56)(7,0.73)(8,0.88)(9,1.04)(10,1.04)};
\addlegendentry{SONCSOCP}
\addplot[color=red,mark=star]
coordinates{(1,0.28)(2,0.31)(3,0.31)(4,0.31)(5,0.43)(6,0.74)(7,0.75)(8,0.74)(9,0.72)(10,0.76)};
\addlegendentry{POEM}
\end{axis}
\end{tikzpicture}
\end{center}
\caption{Running time for the general simplex case}\label{fg1:simplex}
\end{figure}
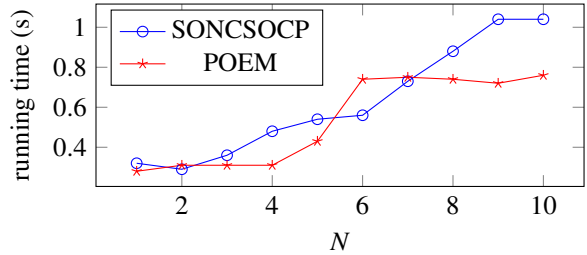

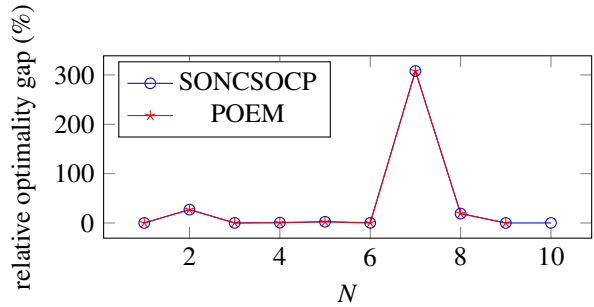
\begin{figure}[htbp]
\begin{center}
\begin{tikzpicture}
\begin{axis}[xlabel={$N$},
ylabel={relative optimality gap (\%)},
legend pos=north west]
\addplot[color=blue,mark=o]
coordinates{(1,0.069)(2,27.13)(3,0.102)(4,0.708)(5,2.46)(6,0.27)(7,308)(8,19.11)(9,0.016)(10,0.143)};
\addlegendentry{SONCSOCP}
\addplot[color=red,mark=star]
coordinates{(1,0.069)(2,27.13)(3,0.102)(4,0.708)(5,2.46)(6,0.27)(7,308)(8,19.11)(9,0.016)};
\addlegendentry{POEM}
\end{axis}
\end{tikzpicture}
\end{center}
\caption{Relative optimality gap for the general simplex case}\label{fg2:simplex}
\end{figure}

\paragraph{\textbf{Arbitrary polytope}}
For the arbitrary polytope case, 
we take $20$ polynomials of different types (labeled by $N$). 
With regard to these instances, 
{\tt POEM} always throws an error ``expected square matrix''. 
Running time and lower bounds obtained with {\tt SONCSOCP} are displayed in Table \ref{tb1:arbitrary}. 
The relative optimality gap is always within $25\%$ and within $1\%$ for $17$ out of $20$ polynomials  (see Figure \ref{fg2:arbitrary}).

\begin{table}[htbp]
{\small
\begin{center}
\begin{tabular}{c|c|cccccccccc}
\multicolumn{2}{c|}{$N$}&$1$&$2$&$3$&$4$&$5$&$6$&$7$&$8$&$9$&$10$\\
\hline
\multicolumn{2}{c|}{$n$}&$10$&$10$&$10$&$10$&$10$&$10$&$10$&$10$&$10$&$10$\\
\multicolumn{2}{c|}{$d$}&$20$&$20$&$20$&$30$&$30$&$30$&$40$&$40$&$40$&$50$\\
\multicolumn{2}{c|}{$t$}&$30$&$100$&$300$&$30$&$100$&$300$&$30$&$100$&$300$&$30$\\
\multicolumn{2}{c|}{$l$}&$15$&$71$&$231$&$15$&$71$&$231$&$15$&$71$&$231$&$15$\\
\hline
\multirow{2}*{{\tt SONCSOCP}}&time&$0.38$&$1.75$&$6.86$&$0.64$&$3.13$&$ 11.3$&$0.72$&$4.01$&$14.6$&$0.76$\\
\cline{2-12}
&opt&$0.70$&$3.32$&$31.7$&$3.31$&$15.3$&$3.31$&$0.47$&$5.42$&$38.7$&$1.56$\\
\hline
\hline
\multicolumn{2}{c|}{$N$}&$11$&$12$&$13$&$14$&$15$&$16$&$17$&$18$&$19$&$20$\\
\hline
\multicolumn{2}{c|}{$n$}&$10$&$10$&$10$&$10$&$10$&$20$&$20$&$20$&$20$&$20$\\
\multicolumn{2}{c|}{$d$}&$50$&$50$&$60$&$60$&$60$&$30$&$30$&$40$&$40$&$40$\\
\multicolumn{2}{c|}{$t$}&$100$&$300$&$30$&$100$&$300$&$50$&$100$&$50$&$100$&$200$\\
\multicolumn{2}{c|}{$l$}&$71$&$231$&$15$&$71$&$231$&$5$&$15$&$5$&$15$&$35$\\
\hline
\multirow{2}*{{\tt SONCSOCP}}&time&$4.41$&$16.8$&$1.84$&$11.2$&$42.4$&$ 3.20$&$8.84$&$2.60$&$10.5$&$38.7$\\
\cline{2-12}
&opt&$0.20$&$7.00$&$3.31$&$2.52$&$23.4$&$0.70$&$4.91$&$4.13$&$2.81$&$9.97$\\
\hline
\end{tabular}
\end{center}}
\caption{Results for the arbitrary polytope case}\label{tb1:arbitrary}
\end{table}


\begin{figure}[htbp]
\begin{center}
\begin{tikzpicture}
\begin{axis}[xlabel={$N$},
ylabel={relative optimality gap (\%)},
legend pos=north west]
\addplot[color=blue,mark=o]
coordinates{(1,1.56)(2,0.0015)(3,0.0006)(4,0.02)(5,0.004)(6,0.002)(7,0.25)(8,0.01)(9,0.0008)(10,0.026)(11,0.023)(12,0.008)(13,0.02)(14,0.07)(15,0.04)(16,0.01)(17,0.43)(18,0.0002)(19,22.9)(20,24.8)};
\addlegendentry{SONCSOCP}
\end{axis}
\end{tikzpicture}
\end{center}
\caption{Relative optimality gap for the arbitrary polytope case}\label{fg2:arbitrary}
\end{figure}
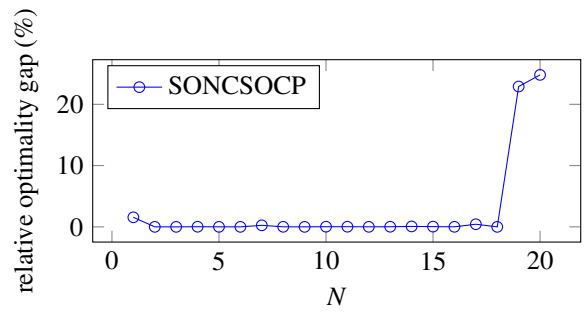

\subsection{Exact nonnegativity certificates}
In this subsection, we certify nonnegativity of polynomials by computing an exact SOBS decomposition via the procedure ${\tt ExactSOBS}$. The benchmarks are selected to be in the interior of the SONC cone. We take the number of variables $n=4,8,10$, the degree $d=10,20,30$ and the number of terms $t=20,30,50,100$. We set the precision parameter $\tilde \delta$ for the SOCP solver (Mosek) to be $10^{-8}$ and the rounding precision $\hat \delta$ to be $10^{-5}$.
The results are displayed in Table \ref{exact:4}, \ref{exact:8} and \ref{exact:10} respectively. For polynomials in each table, their Newton polytope is either a standard simplex or an arbitrary polytope. The notation ``-" indicates that the corresponding polynomial does not exist in the database.

As we can see from the tables, the procedure ${\tt ExactSOBS}$ is very efficient to compute an exact SOBS decomposition. Even for the largest instance ($n=10,d=30,t=100$), it takes around $6$ seconds to get the result. We also observe that (somewhat surprisingly) the time spent for symbolic computation and the time spent for numerical computation are nearly equal.
Note that we were not able to compare our certification procedure with the current version of the ${\tt POEM}$ software library, since the rounding-projection procedure from \cite{ma19} seems to be not available anymore.

\begin{table}[htbp]
{\small
\begin{center}
\begin{tabular}{c|cc|cc|cc||cc|cc|cc}
\multicolumn{7}{c||}{$n=4$, standard simplex}&\multicolumn{6}{c}{$n=4$, arbitrary polytope}\\
\hline
\multirow{2}*{$t$}&\multicolumn{2}{c|}{$d=10$}&\multicolumn{2}{c|}{$d=20$}&\multicolumn{2}{c||}{$d=30$}&\multicolumn{2}{c|}{$d=10$}&\multicolumn{2}{c|}{$d=20$}&\multicolumn{2}{c}{$d=30$}\\
\cline{2-7}\cline{8-13}
&time&bit&time&bit&time&bit&time&bit&time&bit&time&bit\\
\hline
20	&0.06	&123	&0.06	&154	&0.07	&162 &0.07	&109	&0.07	&129	&0.13	&118\\
30	&0.09	&189	&0.09	&193	&0.09	&199 &0.09	&135	&0.12	&106	&0.16	&121\\
50	&0.09	&248	&0.12	&282	&0.14	&332 &0.22	&195	&0.13	&124	&0.35	&226\\
100	&0.16	&362	&0.18	&364	&-	    &- &-	&-	&-	&-	&0.48	&210\\
\end{tabular}
\end{center}}
\caption{Exact SOBS decompositions for polynomials with four variables}\label{exact:4}
\end{table}

\begin{table}[htbp]
{\small
\begin{center}
\begin{tabular}{c|cc|cc|cc||cc|cc|cc}
\multicolumn{7}{c||}{$n=8$, standard simplex}&\multicolumn{6}{c}{$n=8$, arbitrary polytope}\\
\hline
\multirow{2}*{$t$}&\multicolumn{2}{c|}{$d=10$}&\multicolumn{2}{c|}{$d=20$}&\multicolumn{2}{c||}{$d=30$}&\multicolumn{2}{c|}{$d=10$}&\multicolumn{2}{c|}{$d=20$}&\multicolumn{2}{c}{$d=30$}\\
\cline{2-7}\cline{8-13}
&time&bit&time&bit&time&bit&time&bit&time&bit&time&bit\\
\hline
20	&-	&-	&0.07	&135	&0.07	&144		&0.13	&101	&0.27	&138	&0.19	&140\\
30	&-	&-	&0.11	&177	&0.13	&184		&0.17	&141	&0.31	&143	&0.56	&172\\
50	&-	&-	&0.18	&273	&0.18	&279		&0.29	&144	&0.75	&212	&1.28	&206\\
100	&-	&-	&0.45	&356	&0.46	&419	&0.81	&277	&2.61	&333	&3.45	&264\\
\end{tabular}
\end{center}}
\caption{Exact SOBS decompositions for polynomials with eight variables}\label{exact:8}
\end{table}

\begin{table}[htbp]
{\small
\begin{center}
\begin{tabular}{c|cc|cc|cc||cc|cc|cc}
\multicolumn{7}{c||}{$n=10$, standard simplex}&\multicolumn{6}{c}{$n=10$, arbitrary polytope}\\
\hline
\multirow{2}*{$t$}&\multicolumn{2}{c|}{$d=10$}&\multicolumn{2}{c|}{$d=20$}&\multicolumn{2}{c||}{$d=30$}&\multicolumn{2}{c|}{$d=10$}&\multicolumn{2}{c|}{$d=20$}&\multicolumn{2}{c}{$d=30$}\\
\cline{2-7}\cline{8-13}
&time&bit&time&bit&time&bit&time&bit&time&bit&time&bit\\
\hline
20	&-	&-	&0.06	&119	&0.07	&124		&-	&-	&0.29	&158	&0.58	&186\\
30	&-	&-	&0.16	&173	&0.13	&185		&0.11	&140	&0.84	&169	&1.46	&175\\
50	&-	&-	&0.24	&253	&0.25	&256		&0.31	&178	&1.43	&179	&2.77	&224\\
100	&-	&-	&0.53	&431	&0.44	&401		&-	&-	&3.98	&307	&6.38	&302\\
\end{tabular}
\end{center}}
\caption{Exact SOBS decompositions for polynomials with ten variables}\label{exact:10}
\end{table}

\section{Conclusions}
\label{sec:concl}
In this paper, we provide a constructive proof that each SONC cone admits a SOC representation. 
Based on this, we propose an algorithm to compute a lower bound for unconstrained POPs via SOCP. 
Numerical experiments demonstrate the efficiency of our algorithm even when the number of variables and the degree are fairly large. 
Even though the complexity of our algorithm depends on the degree in theory, 
it turns out that this dependency is rather mild. 
For all numerical examples tested in this paper, the running time is below one minute even for polynomials of degree up to $60$.

Since the running time is satisfactory, the main concern of SONC-based algorithms for sparse polynomial optimization might be the quality of obtained lower bounds. 
For many examples tested in this paper, the relative optimality gap is within $1\%$. However, 
it can happen that the SONC lower bound is not accurate and this cannot be avoided by computing an optimal simplex cover. 
To improve the quality of such bounds, it is mandatory to find more complex representations of nonnegative polynomials, which involve SONC polynomials. We will investigate it in the future.

Another line of research is to extend our SONC-SOCP framework to constrained polynomial optimization. Note that constrained polynomial optimization based on the SAGE certificate has been studied in \cite{murray2020signomial}. It is also worth investigating how the SOCP methodology works in the SAGE context.


\paragraph{\textbf{Acknowledgements}} 
Both authors were supported by the Tremplin ERC Stg Grant ANR-18-ERC2-0004-01 (T-COPS project).
The second author was supported by the FMJH Program PGMO (EPICS project) and  EDF, Thales, Orange et Criteo.
This work has benefited from  the European Union's Horizon 2020 research and innovation programme under the Marie Sklodowska-Curie Actions, grant agreement 813211 (POEMA) as well as from the AI Interdisciplinary Institute ANITI funding, through the French ``Investing for the Future PIA3'' program under the Grant agreement n$^{\circ}$ANR-19-PI3A-0004.
\bibliographystyle{plainnat}
\bibliography{refer}

\end{document}